\documentclass[12pt]{amsart}
\advance\hoffset by -0.5 truecm
\usepackage{amsmath,amscd}
\usepackage{amssymb}
\usepackage{amsthm}
\usepackage{array}
\usepackage{graphicx}
\usepackage{color}

\usepackage[colorinlistoftodos]{todonotes}

\usepackage[T1]{fontenc}
\usepackage[utf8]{inputenc}
\usepackage[french,english]{babel}

\makeatletter
\newenvironment{abstracts}{%
  \ifx\maketitle\relax
    \ClassWarning{\@classname}{Abstract should precede
      \protect\maketitle\space in AMS document classes; reported}%
  \fi
  \global\setbox\abstractbox=\vtop \bgroup
    \normalfont\Small
    \list{}{\labelwidth\z@
      \leftmargin3pc \rightmargin\leftmargin
      \listparindent\normalparindent \itemindent\z@
      \parsep\z@ \@plus\p@
      
      \itemsep\medskipamount
    }%
}{%
  \endlist\egroup
  \ifx\@setabstract\relax \@setabstracta \fi
}

\newcommand{\abstractin}[1]{%
  \otherlanguage{#1}%
  \item[\hskip\labelsep\scshape\abstractname.]%
}
\makeatother

\usepackage{mathrsfs}
\usepackage{hyperref}

\usepackage{comment}

\usepackage{calc}
             {\begin{list}{\arabic{enumi}.}{\usecounter{enumi}%
              \setlength{\labelsep}{0.5em}%
              \settowidth{\labelwidth}{\arabic{enumi}.}%
              \setlength{\leftmargin}{\labelwidth+\labelsep}}}%
             {\end{list}}

\newtheorem{Theorem}{Theorem}[section]

\newtheorem{Lemma}[Theorem]{Lemma}

\newtheorem{Proposition}[Theorem]{Proposition}

\theoremstyle{definition}

\newtheorem{Definition}[Theorem]{Definition}

\newtheorem{Remark}[Theorem]{Remark}

\numberwithin{equation}{section}

\def \dim{{\mbox {dim}}\,}

\def\V{\mbox{Var}}

\def\R\re
\def\V{\bf V}

\def \re{{\mathbb R}}
\def \mR{{\mathbb R}}

\def \C{{\mathbb C}}
\def \mC{{\mathbb C}}

\def \V{{\bf V}}

\def \l{\ell}

\newcommand{\mc}{\mathcal}

\newcommand{\id}{\mathrm{Id}}

\renewcommand{\dim}{\mathrm{dim}}

\newcommand{\abs}[1]{\lvert #1 \rvert}
\newcommand{\norm}[1]{\lVert #1 \rVert}
\newcommand{\br}[1]{\langle #1 \rangle}

\newcommand{\eps}{\varepsilon}

\newcommand{\ol}[1]{\overline{#1}}

\def \vd{\overset{\tt{v}}{\nabla}}
\def \hd{\overset{\tt{h}}{\nabla}}
\def \vdiv{\overset{\tt{v}}{\mbox{\rm div}}}
\def \hdiv{\overset{\tt{h}}{\mbox{\rm div}}}

\def \GL{\operatorname{GL}}

\newcounter{sidenote}
\begin{document}

\title[Carleman estimates for geodesic X-ray transforms]{Carleman estimates for geodesic X-ray transforms} 

\author[G.P. Paternain]{Gabriel P. Paternain}
\address{ Department of Pure Mathematics and Mathematical Statistics,
University of Cambridge,
Cambridge CB3 0WB, UK}
\email {g.p.paternain@dpmms.cam.ac.uk}

\author[M. Salo]{Mikko Salo}
\address{Department of Mathematics and Statistics, University of Jyv\"askyl\"a}
\email{mikko.j.salo@jyu.fi}




\begin{abstracts}
\abstractin{english}In this article we introduce an approach for studying the geodesic X-ray transform and related geometric inverse problems by using Carleman estimates. The main result states that on compact negatively curved manifolds (resp.\ nonpositively curved simple or Anosov manifolds), the geodesic vector field satisfies a Carleman estimate with logarithmic weights (resp.\ linear weights) on the frequency side. As a particular consequence, on negatively curved simple manifolds the geodesic X-ray transform with attenuation given by a general connection and Higgs field is invertible modulo natural obstructions. The proof is based on showing that the Pestov energy identity for the geodesic vector field completely localizes in frequency. Our approach works in all dimensions $\geq 2$, on negatively curved manifolds with or without boundary, and for tensor fields of any order.

\end{abstracts}

\maketitle

\tableofcontents

\section{Introduction}

\noindent {\bf Motivation.} 
The geodesic X-ray transform is a central object in geometric inverse problems. In Euclidean space it reduces to the standard X-ray transform, which encodes the integrals of a function over straight lines and provides the mathematical model for imaging methods such as X-ray CT and PET \cite{Natterer}. If the underlying medium is not Euclidean, one needs to consider more general curve families. On a Riemannian manifold the geodesic curves provide a natural candidate, and the geodesic X-ray transform encodes the integrals of a function over geodesics. This transform plays an important role (often via linearization or pseudo-linearization arguments) in inverse problems such as 
\begin{itemize}
\item 
boundary and lens rigidity on manifolds with boundary \cite{PestovUhlmann, Guillarmou, SUV_full};
\item 
spectral rigidity on closed manifolds with Anosov geodesic flow \cite{PSU3};
\item 
inverse boundary value problems for hyperbolic equations \cite{StefanovYang};
\item 
inverse boundary value problems for elliptic equations (Calder\'on problem) \cite{DKSaU, DKLS}.
\end{itemize}

Several approaches have been introduced for studying the X-ray transform:
\begin{itemize}
\item 
direct methods such as Fourier analysis in symmetric geometries \cite{Helgason};
\item 
microlocal methods based on interpreting the transform or its normal operator as a Fourier integral or pseudodifferential operator \cite{SU3, SU09, UhlmannVasy};
\item 
reductions to PDE and energy methods \cite{Sh, PSU1, GPSU}.
\end{itemize}
In particular, the works \cite{PSU1, GPSU, UhlmannVasy} involve various $L^2$ estimates with exponential weights. This suggests that the Carleman estimate methodology, which relies on exponentially weighted $L^2$ estimates and provides a very powerful general approach to uniqueness results for linear PDE \cite[Chapter XXVIII]{Hormander}, could have consequences for the geodesic X-ray transform as well. However, as explained in Section \ref{sec_analogies}, the PDE related to the X-ray transform are nonstandard and thus the existing theory of Carleman estimates cannot be directly applied.

This paper presents the first steps toward a Carleman estimate approach for the geodesic X-ray transform. We prove two Carleman estimates for the geodesic vector field, one with logarithmic and another with linear Carleman weights, that have the right form in order to be applied to the geodesic X-ray transform. As a consequence we obtain invertibility results for certain weighted X-ray transforms that correspond to large lower order perturbations in the related PDE, and uniqueness results in related problems such as scattering rigidity and transparent connections. Our approach works in all dimensions $\geq 2$, on manifolds with or without boundary, and for tensor fields of any order. However, the main Carleman estimate requires that the manifold has negative curvature, and the Carleman weights are purely on the frequency side (one can think of them as pseudodifferential weights) and thus the estimates do not currently lead to local results for the X-ray transform on the space side as in \cite{UhlmannVasy}.

\vspace{15pt}

\noindent {\bf Main estimates.}
Let $(M,g)$ be a compact oriented Riemannian manifold with or without boundary, and let $X$ be the geodesic vector field regarded as a first order differential operator $X:C^{\infty}(SM)\to C^{\infty}(SM)$ acting on functions on the unit sphere bundle $SM$. Our first result is a new energy estimate for $X$ when $g$ is negatively curved. 
The new estimate involves polynomial weights on the Fourier side and it can be regarded as a {\it Carleman estimate} for the transport equation $Xu=f$ (see below for analogies with elliptic operators). To describe it we recall some basic harmonic analysis on $SM$. 
By considering the vertical Laplacian $\Delta$ on each fibre $S_{x}M$ of $SM$ we have a natural $L^2$-decomposition $L^{2}(SM)=\oplus_{m\geq 0}H_m(SM)$
into vertical spherical harmonics. We set $\Omega_m:=H_{m}(SM)\cap C^{\infty}(SM)$. Then a function $u$ belongs to $\Omega_m$ if and only if
$-\Delta u=m(m+d-2)u$ where $d=\mathrm{dim}(M)$ (for details see \cite{GK2,DS}). If $u\in L^{2}(SM)$, this decomposition will be written as
\[u=\sum_{m=0}^{\infty}u_{m}, \;\;\;\;\;\;\;u_{m}\in H_{m}(SM).\]
We say that $u$ has \emph{finite degree} if the above sum is finite.

Here is our first main result (the norms are $L^2(SM)$ norms).

\begin{Theorem} \label{thm_carleman_negativecurvature_intro}
Let $(M,g)$ be a compact Riemannian manifold with sectional curvature $\leq -\kappa$ where $\kappa > 0$. Let also $\phi_{\ell} = \log(\ell)$. For any $\tau \geq 1$ and $m \geq 1$, one has 
\begin{align*}
\sum_{\ell=m}^{\infty} e^{2\tau \phi_{\ell}} \norm{u_\ell}^2 \leq \frac{C}{\kappa \tau} \sum_{\ell=m+1}^{\infty} e^{2\tau \phi_\ell}\norm{(Xu)_{\ell}}^2
\end{align*}
whenever $u \in C^{\infty}(SM)$ (with $u|_{\partial(SM)} = 0$ in the boundary case), where $C$ is a positive constant depending only on the dimension of $M$.
\end{Theorem}

We have a similar estimate for compact manifolds with nonpositive curvature, provided that they are \emph{simple} (simply connected with strictly convex boundary and no conjugate points) or \emph{Anosov} (no boundary, and the geodesic flow satisfies the Anosov property). Both the simple and Anosov conditions can be seen as slightly strengthened forms of the condition that no geodesic has conjugate points (see \cite{PSU3}). However, if the curvature is only nonpositive, the logarithmic weights above need to be replaced by stronger linear weights. We remark that both logarithmic and linear weights have been prominent in the theory of Carleman estimates, see e.g. \cite{JerisonKenig, KenigRuizSogge}.

\begin{Theorem} \label{thm_pestov_carleman_nonpositivecurvature_intro}
Let $(M,g)$ be a simple/Anosov Riemannian manifold having nonpositive sectional curvature. There exist $m_0, \tau_0, \kappa > 0$ such that for any $\tau \geq \tau_0$, for any $m \geq m_0$ and for any $u \in C^{\infty}(SM)$ of finite degree (with $u|_{\partial(SM)} = 0$ in the boundary case) one has 
\[
\sum_{\l=m}^{\infty} e^{2\tau \phi_\l} \norm{u_\l}^2 \leq \frac{24}{\kappa e^{2 \tau}} \sum_{\l=m+1}^{\infty} e^{2\tau \phi_\l} \norm{(Xu)_\l}^2
\]
where $\phi_\l = \l$.
\end{Theorem}

Both theorems above are based on weighted frequency localized versions of the \emph{Pestov energy identity}, which has been the main tool in energy methods for X-ray transforms (see e.g.\ \cite{Sh, PSU1, PSU_hd}). The Pestov identity has been used in many forms. A powerful recent variant is the inequality proved in \cite{PSU_hd}, 
\begin{equation} \label{intro_pestov_inequality}
\norm{\nabla_{SM} u} \leq C \norm{\vd X u},
\end{equation}
which is valid on any compact simple/Anosov manifold for any $u \in C^{\infty}(SM)$ (with $u|_{\partial(SM)} = 0$ in the boundary case). Here $\vd$ is the vertical gradient (see Section \ref{sec_preliminaries}). So far the Pestov identity has been expressed in terms of $L^2$ norms, and it has not been known if it is possible to shift the related estimates to different Sobolev scales. However, Theorem \ref{thm_carleman_negativecurvature_intro} is actually a shifted version of the Pestov identity with respect to the mixed norms 
\[
\norm{u}_{L^2_x H^s_v} = \left( \sum_{\l=0}^{\infty} \br{\l}^{2s} \norm{u_\l}^2 \right)^{1/2}
\]
where $s \in \mR$ and $\br{\l} = (1+\l^2)^{1/2}$. Similar mixed norms for $\nabla_{SM} u$ and $\vd Xu$ are defined in Section \ref{sec_carleman_negative_curvature}. Then, the inequality \eqref{intro_pestov_inequality} can be expressed as 
\[
\norm{\nabla_{SM} u}_{L^2_x H^0_v} \leq C \norm{\vd X u}_{L^2_x H^0_v}.
\]
The next result is a shifted version of this inequality.

\begin{Theorem} \label{thm_pestov_carleman_shifted_intro}
Let $(M,g)$ be a compact Riemannian manifold with sectional curvature $\leq -\kappa$ where $\kappa > 0$. For any $s > -1/2$ there is $C = C_{d,s,\kappa} > 0$ such that 
\[
\norm{\nabla_{SM} u}_{L^2_x H^s_v} \leq C \norm{\vd Xu}_{L^2_x H^s_v}
\]
for any $u \in C^{\infty}(SM)$ (with $u|_{\partial(SM)} = 0$ in the boundary case).
\end{Theorem}

\vspace{5pt}

\noindent {\bf Applications.}

\vspace{5pt}

\noindent {\it Uniqueness results for the transport equation.} 
The significance of Theorem \ref{thm_carleman_negativecurvature_intro} is best illustrated by considering the transport equation 
\[
Xu=-f \text{ in $SM$}, \qquad u|_{\partial(SM)} = 0,
\]
where $f$ has \emph{finite degree $m$} (i.e.\ $f_{\l}=0$ for $\l\geq m+1$). Then Theorem \ref{thm_carleman_negativecurvature_intro} implies that any solution $u \in C^{\infty}(SM)$ must satisfy $u_{\l}=0$ for $\l\geq m$ ($u$ has finite degree $m-1$) at least when $m \geq 1$.
This fact provides a full solution to the tensor tomography problem in negative curvature (see \cite{PS, CS} and \cite{PSU4} for a recent survey).
The case $m=2$ is at the heart of the proof that closed negatively curved manifolds are spectrally rigid \cite{CS,GK}.
Given an isospectral deformation $g_{\varepsilon}$ of the metric $g$ it was shown in \cite{GK} that the symmetric 2-tensor $f:=\frac{\partial g_{\varepsilon}}{\partial\varepsilon}|_{\varepsilon=0}$ must have the property that it integrates to zero along every closed geodesic of $g$. Since the geodesic flow of $g$ is Anosov an application of the Livsic theorem  \cite{dLMM} shows that there is a smooth function $u$ such that $Xu=\tilde{f}\in \Omega_{0}\oplus\Omega_{2}$, where $\tilde{f}$ is the natural function induced on $SM$ by the symmetric 2-tensor $f$. The Carleman estimate in Theorem \ref{thm_carleman_negativecurvature_intro} implies that $u=u_{1}$ and this uniquely determines a vector field $V$ on $M$. Doing this for every $\varepsilon$ produces a parameter dependent vector field $V_{\varepsilon}$ whose flow makes the deformation $g_{\varepsilon}$ trivial, thus showing spectral rigidity.

However, the estimate in Theorem \ref{thm_carleman_negativecurvature_intro} is considerably more powerful since we can in principle adjust the parameter $\tau$ at will as long as we know that $u\in C^{\infty}(SM)$.
Here is the typical application we have in mind. Suppose that $\Phi \in C^{\infty}(M, \mC^{n\times n})$ is a smooth $n\times n$ matrix-valued potential in $M$ (we call $\Phi$ a \emph{general Higgs field}). Consider the attenuated transport equation
\begin{equation}
(X+\Phi)u=-f_{0}\in\Omega_0,   \qquad u|_{\partial(SM)} = 0
\label{eq:transport1}
\end{equation}
where $u\in C^{\infty}(SM,\mC^n)$.  The objective is to show that $f_0=0$; this is precisely the uniqueness problem that arises when trying to prove injectivity of the attenuated geodesic X-ray transform where the attenuation is given
by the matrix $\Phi$ (see below).
Since $(\Phi u)_{\l}=\Phi  u_{\l}$, we see using \eqref{eq:transport1} that 
\[
\norm{(Xu)_{\l}}\leq C\norm {u_{\l}}, \qquad \l \geq 1,
\]
 where $C = \norm{\Phi}_{L^{\infty}(M)}$. If we input this information into Theorem \ref{thm_carleman_negativecurvature_intro}
and we take $\tau$ sufficiently large  (depending on $\norm{\Phi}_{L^{\infty}(M)}$) we deduce that 
$u=u_0\in\Omega_0$. Since $Xu_0\in\Omega_1$, using \eqref{eq:transport1} we obtain $f_0=0$ as desired.

However, the Carleman estimate can also deal with other matrix attenuations like {\it connections}.
For our purposes these are simply $n\times n$ matrices $A$ of complex valued 1-forms, i.e.\ $A \in C^{\infty}(M, (\Lambda^1 M)^{n \times n})$. We call such geometric objects \emph{general connections} on $M$, and they can be naturally considered as smooth functions $A:SM\to \mC^{n\times n}$ such that $A\in\Omega_{1}$. More generally, the Carleman estimate even makes it possible to include nonlinear attenuation terms in the transport equation. Here is a general result in this direction.

\begin{Theorem} \label{thm_xray_transport_intro}
Let $(M,g)$ be a compact Riemannian manifold with sectional curvature $\leq -\kappa$ where $\kappa > 0$. Let $\l_0 \geq 2$, and suppose that $\mathcal{A}: C^{\infty}(SM, \mC^n) \to C^{\infty}(SM, \mC^n)$ is a map such that 
\[
\norm{(\mathcal{A}(u))_\l} \leq C ( \norm{u_{\l-1}} + \norm{u_\l} + \norm{u_{\l+1}} ), \qquad \l \geq \l_0.
\]
If $f \in C^{\infty}(SM, \mC^n)$ has finite degree and if $u \in C^{\infty}(SM, \mC^n)$ satisfies 
\[
Xu + \mathcal{A}(u) = -f \text{ in $SM$}, \qquad u|_{\partial(SM)} = 0,
\]
then $u$ has finite degree.
\end{Theorem}

The previous result can be interpreted as a fundamental uniqueness, or vanishing, theorem for the attenuated transport equation. In \cite{GPSU}, such a result was proved on negatively curved manifolds in the special (linear) case 
\[
\mathcal{A}(u) = (A+\Phi)u
\]
where $A$ and $\Phi$ are a \emph{skew-Hermitian} connection and Higgs field (when $A$ is skew-Hermitian, the connection is also referred to as {\it unitary} because the underlying structure group is the unitary group). The skew-Hermitian condition implies that the transport operator $X + A + \Phi$ is skew-symmetric, which led to improved energy estimates in \cite{GPSU}. In our case, where $A$ and $\Phi$ can be general matrices, there is a loss of symmetry and large error terms appear in the energy estimates. The Carleman estimate in Theorem \ref{thm_carleman_negativecurvature_intro} is required to deal with these large errors (see Section \ref{sec_analogies} for a more detailed discussion).

In the more general case where $(M,g)$ is a compact simple/Anosov manifold with nonpositive sectional curvature, one could use Theorem \ref{thm_pestov_carleman_nonpositivecurvature_intro} to prove an analogue of Theorem \ref{thm_xray_transport_intro} with two additional restrictions: the map $\mathcal{A}$ should satisfy 
\[
\norm{(\mathcal{A}(u))_\l} \leq C \norm{u_\l}, \qquad \l \geq \l_0,
\]
and the solution $u \in C^{\infty}(SM, \mC^n)$ should additionally satisfy 
\begin{equation} \label{exponential_decay_regularity}
\lim_{\l \to \infty} e^{\tau \l} \norm{X_- u_\l} = 0,
\end{equation}
where $X_-$ is the part of $X$ that maps $\Omega_m$ to $\Omega_{m-1}$ (see Section \ref{sec_preliminaries}) and $\tau>0$ is sufficiently large (depending on $\mathcal A$). The condition for $\mathcal{A}$ is satisfied if $\mathcal{A}$ is a general Higgs field, but \eqref{exponential_decay_regularity} should be viewed as an additional regularity condition (it states that the Fourier coefficients of $X_- u$ decay exponentially, which corresponds to real-analyticity of $X_- u$ in the $v$ variable). Thus, the Carleman estimate in Theorem \ref{thm_pestov_carleman_nonpositivecurvature_intro} shows that the injectivity result for a general Higgs field reduces to the regularity result \eqref{exponential_decay_regularity}. 
However, it is not at all clear how to establish \eqref{exponential_decay_regularity}.  For simple surfaces, it is possible to bypass these difficulties using a loop group factorization result and reducing the problem to the skew-hermitian case, see \cite{PS_LoopG}.


Theorem \ref{thm_xray_transport_intro} easily implies uniqueness results for X-ray transforms, whenever one has a regularity result showing that the vanishing of the X-ray transform yields a $C^{\infty}$ solution $u$ of a transport equation. Let us next discuss this in more detail.

\vspace{5pt}

\noindent {\it The attenuated X-ray transform.} Let $(M,g)$ be a compact Riemannian manifold with smooth boundary. We assume that $(M,g)$ is \emph{nontrapping}, i.e.\ for any $(x,v) \in SM$ the geodesic $\gamma_{x,v}(t)$ through $(x,v)$ exits $M$ at some finite time $\tau(x,v) \geq 0$. Given $a \in C^{\infty}(M)$ (scalar attenuation), the attenuated geodesic X-ray transform of $f$ is the map 
\[
I_a f: \partial_+(SM) \to \mC, \ \ I_a f(x,v) = \int_0^{\tau(x,v)} f(\gamma_{x,v}(t)) \,e^{\int_0^{t}a(\gamma_{x,v}(s)) \,ds} \,dt,
\]
where $\partial_{\pm}(SM) = \{ (x,v) \in SM \,;\, x \in \partial M, \pm \langle v, \nu \rangle \leq 0 \}$ are the inward and outward pointing boundaries of $SM$ (with $\nu$ being the outer unit normal to $\partial M$). It is easy to see that $I_a f$ can be equivalently defined as 
\[
I_a f = u|_{\partial_+(SM)}
\]
where $u$ is the solution of 
\[
(X+a)u = -f \text{ in $SM$}, \qquad u|_{\partial_-(SM)} = 0.
\]

More generally, we can consider the vector-valued case where $f \in C^{\infty}(SM, \mC^n)$, and the attenuation is given by a general connection $A \in C^{\infty}(M, (\Lambda^1 M)^{n \times n})$ and a general Higgs field $\Phi \in C^{\infty}(M, \mC^{n \times n})$. The matrix of $1$-forms $A$ can be identified with the connection $\nabla = d + A$ on the trivial bundle $\mathcal{E} = M \times \mC^n$. In this case, the attenuated X-ray transform of $f$ is defined as 
\[
I_{A+\Phi} f = u|_{\partial_+(SM)},
\]
where $u$ is the solution of 
\[
(X + A + \Phi) u = -f \text{ in $SM$}, \qquad u|_{\partial_-(SM)} = 0.
\]

The basic inverse problem for $0$-tensors (resp.\ $m$-tensors) is the following: given $f \in C^{\infty}(M, \mC^n)$ (resp.\ $f \in C^{\infty}(SM, \mC^n)$ with degree $m$) such that $I_{A,\Phi} f = 0$, is it true that $f \equiv 0$ (resp.\ $f = -(X+A+\Phi)u$ for some $u \in C^{\infty}(SM, \mC^n)$ of degree $m-1$ with $u|_{\partial(SM)} = 0$)? The following theorem gives an affirmative answer on nontrapping negatively curved manifolds.

\begin{Theorem} \label{thm_intro_attenuated_xray}
Let $(M,g)$ be a compact Riemannian nontrapping negatively curved manifold with $C^{\infty}$ boundary, let $A$ be a general connection, and let $\Phi$ be a general Higgs field in $M$. Let $f \in C^{\infty}(SM, \mC^n)$, and assume either that $M$ has strictly convex boundary, or that $f$ is supported in the interior of $SM$. If $f$ has degree $m \geq 0$ and if the attenuated X-ray transform of $f$ vanishes (i.e.\ $I_{A+\Phi} f = 0$), then 
\[
f = -(X+A+\Phi) u
\]
for some $u \in C^{\infty}(SM,\mC^n)$ with degree $m-1$ such that $u|_{\partial(SM)} = 0$.
\end{Theorem}

This kind of result was proved for skew-Hermitian $A$ and $\Phi$ in \cite{GPSU} when $M$ is negatively curved with strictly convex boundary (in this case there is a stronger regularity result based on the work \cite{Guillarmou}, showing that one can drop the nontrapping assumption). Also, for strictly convex manifolds with $\dim(M) \geq 3$ a stronger result was proved in \cite{PSUZ16}, based on the method of \cite{UhlmannVasy}. Genericity results for simple manifolds are established in \cite{Zhou2017} and reconstruction formulas for simple surfaces (up to a Fredholm error) are given in \cite{MP18}.
The result for nonconvex boundaries is similar to \cite{Da06, GMT} which consider the unattenuated case. 
Theorem \ref{thm_intro_attenuated_xray} is new if the boundary is nonconvex.

\vspace{5pt}

\noindent {\it Scattering rigidity and transparent pairs.} 
Finally, we briefly discuss results for nonlinear geometric inverse problems that can be reduced to Theorems \ref{thm_xray_transport_intro} and \ref{thm_intro_attenuated_xray} via pseudo-linearization arguments. Given a general connection $A$ and a Higgs field $\Phi$ on the trivial bundle $M \times \C^n$, the \emph{scattering data} for $(A,\Phi)$ is the map 
\[
C_{A,\Phi}:\partial_{-}(SM)\to \GL(n,\C), \ \  C_{A,\Phi} = U|_{\partial_{-}(SM)},
\]
where $U:SM\to \GL(n,\C)$ is the unique matrix solution of the transport equation
\[
(X+A+\Phi) U = 0 \text{ in } SM, \quad U|_{\partial_{+}(SM)}=\id.
\]
The data $C_{A,\Phi}$ corresponds to boundary measurements for parallel transport if $(M,g)$ is known (it measures how vectors are transformed when they are parallel transported with respect to $(A,\Phi)$ along geodesics between two boundary points). The map $(A,\Phi)\mapsto C_{A,\Phi}$ is sometimes called the non-abelian Radon transform, or the X-ray transform for a non-abelian connection and Higgs field (see \cite{PSU2, Novikov_nonabelian}). The data $C_{A,\Phi}$ has a natural gauge invariance under a change of basis, i.e., 
\[
C_{Q^{-1}(X+A)Q,Q^{-1} \Phi Q} = C_{A,\Phi} \quad \text{if $Q \in C^{\infty}(M,\GL(n,\C))$ satisfies $Q|_{\partial M} = \id$}.
\]
It follows that from the knowledge of $C_{A,\Phi}$ one can only expect to recover $A$ and $\Phi$ up to a gauge transformation via $Q$ which satisfies $Q|_{\partial M} = \id$. Our next result gives the following positive answer:

\begin{Theorem} \label{thm_intro_scattering_rigidity}
Let $(M,g)$ be a compact Riemannian negatively curved nontrapping manifold, let
$A$ and $B$ be two general connections, and let $\Phi$ and $\Psi$ be two general Higgs fields.
Assume either that $M$ has strictly convex boundary, or that the pairs $(A,\Phi)$, $(B,\Psi)$ are supported in the interior of $M$.
Then $C_{A,\Phi}=C_{B,\Psi}$ implies that there exists $Q \in C^{\infty}(M,\GL(n,\C))$ such that $Q|_{\partial M}=\id$ and
$B=Q^{-1}dQ+Q^{-1}AQ$, $\Psi = Q^{-1} \Phi Q$.
\end{Theorem}

Again, when $M$ has strictly convex boundary, this was proved for skew-Hermitian connections and Higgs fields in \cite{GPSU} (also without the non-trapping assumption, based on the regularity theory of \cite{Guillarmou}), and for a more general class of manifolds with $\dim(M) \geq 3$ in \cite{PSUZ16}. 
In the Euclidean case the problem has been extensively studied, cf. \cite{E,FU,No,Novikov_nonabelian}.

An analogous result can be formulated on closed negatively curved manifolds without boundary. In this case, we show that if $A$ is a general connection and $\Phi$ is a general Higgs field in $M$, and if the parallel transport for $(A,\Phi)$ along periodic geodesics is the identity map, then $(A,\Phi)$ is gauge equivalent to the trivial pair $(0,0)$ (i.e.\ there are no \emph{transparent pairs}) unless there is an obstruction given by twisted conformal Killing tensors. We refer to Section \ref{sec_applications} for more details.

\vspace{15pt}

\noindent {\bf Methods and analogies.}
The starting point of the Carleman estimate is a well-known $L^2$ energy identity on $SM$ called
{\it Pestov identity}.  If $u \in C^{\infty}(SM)$, there is a splitting (induced by the Sasaki metric and the Levi-Civita connection)
\[
\nabla_{SM} u = \underbrace{(Xu) X + \hd u}_{\textnormal{x-derivatives}} + \underbrace{\vd u}_{\textnormal{v-derivatives}}.
\]
The basic energy identity for $P:=\vd X$ reads \cite[Proposition 2.2]{PSU_hd}
\begin{equation}
\norm{Pu}^2 = ( ( -X^2 - R)\vd u, \vd u) + (d-1) \norm{Xu}^2
\label{eq:pestov}
\end{equation}
where $R(x,v):\{v\}^{\perp}\to \{v\}^{\perp}$ is the operator determined by the Riemann curvature tensor $R$ of $(M,g)$ by $R(x,v)w=R_{x}(w,v)v$. An important observation of the present paper is that
\eqref{eq:pestov} {\it localizes in frequency}. By this we mean that it is possible to recover \eqref{eq:pestov} by summing in $\l$ the identity \eqref{eq:pestov} applied to functions $u\in \Omega_{\l}$. This was observed in \cite{PSU_hd} for $d=2$, but here we prove it for all dimensions. This fact paves the way for the proof of the Carleman estimates: by multiplying the frequency localized estimates by suitable weights, adding them up and using negative curvature to absorb errors we are able to derive the desired inequality.
The localization in frequency has recently been applied in \cite{PS_MZ} to give a sharp stability estimate for tensor tomography in non-positive curvature.

There is an interesting analogy between the proof of Theorem \ref{thm_intro_attenuated_xray}, which is related to the X-ray transform, and the \emph{absence of embedded eigenvalues} for Schr\"odinger operators. If $H = -\Delta + V - \lambda$ is a Schr\"odinger operator in $\mR^d$ where $V$ is a short-range potential (i.e.\ $\abs{V(x)} \leq C(1+\abs{x})^{-1-\eps}$ for some $\eps > 0$ and for a.e.\ $x \in \mR^d$) and $\lambda > 0$ is a positive energy, the absence of embedded eigenvalues states that any solution $u \in L^2(\mR^d)$ of $Hu = 0$ must satisfy $u \equiv 0$ (see e.g.\ \cite[Section 14.7]{Hormander}). A standard proof of this fact proceeds in three steps:
\begin{enumerate}
\item[1.]
\emph{Rapid decay}: any solution $u$ decays rapidly at infinity.
\item[2.]
\emph{Unique continuation at infinity}: any solution $u$ that decays rapidly at infinity must vanish outside a compact set.
\item[3.]
\emph{Weak unique continuation}: any solution $u$ that vanishes outside a compact set must be identically zero.
\end{enumerate}
Our proof of Theorem \ref{thm_intro_attenuated_xray} follows the exact same pattern, even though the equation is very different. The rapid decay result is replaced by a regularity result from \cite{PSU2}, unique continuation at infinity is provided by Theorem \ref{thm_carleman_negativecurvature_intro}, and weak unique continuation is replaced by the absence of twisted conformal Killing tensors proved in \cite{GPSU}. 

As a final remark, the uniqueness problems studied in this article are actually closer to long-range scattering or the Landis conjecture rather than short-range scattering. For instance, a Higgs field $\Phi$ can be interpreted as a ''potential'' $\mathcal{V}: u \mapsto \Phi u$ in the equation $(X+\Phi)u = -f$. In terms of Fourier coefficients, one has 
\[
\norm{(\mathcal{V}(u))_\l} \leq C \norm{u_\l}.
\]
This means that the potential $\mathcal{V}$ is bounded but has no decay as $\l \to \infty$. The analogue of a short-range condition would be $\norm{(\mathcal{V}(u))_\l} \leq \br{\l}^{-1-\eps} \norm{u_\l}$, and an analogue of a long-range condition would be $\norm{(\mathcal{V}(u))_\l} \leq C \br{\l}^{-\eps} \norm{u_\l}$ for some $\eps > 0$ (and a condition for derivatives, see \cite[Chapter XXX]{Hormander}). The fact that general connections and Higgs fields lead to potentials having no decay as $\l \to \infty$, as in the Landis conjecture \cite{KenigSilvestreWang}, makes the uniqueness questions studied in this article challenging.

\vspace{15pt}


\noindent {\bf Organization of the paper.} Section \ref{sec_preliminaries} contains geometric preliminaries on the unit sphere bundle and vertical spherical harmonics. Section \ref{sec_analogies} discusses in more detail certain useful analogies between the geodesic vector field $X$ and the Laplace operator that have in part motivated this paper. Section \ref{section_pestov_general_connection} describes a general Pestov identity with connection and Section \ref{sec_frequency_localization} contains the key result on frequency localization mentioned above. Section \ref{sec_carleman_negative_curvature} establishes the Carleman estimate in Theorem \ref{thm_carleman_negativecurvature_intro}, and contains a proof Theorem \ref{thm_pestov_carleman_shifted_intro} which interprets the Carleman estimate as a version of the Pestov identity \eqref{eq:pestov}
that has been shifted to a different regularity scale.

Section \ref{sec_nonpositive_curvature} proves the Carleman estimate in nonpositive curvature, Theorem \ref{thm_pestov_carleman_nonpositivecurvature_intro}. 
Section \ref{sec_regularity} contains a regularity result for the transport equation in the case when the boundary is not strictly convex. 
Finally, Section \ref{sec_applications} contains all the applications and the proofs of
Theorems \ref{thm_xray_transport_intro}--\ref{thm_intro_scattering_rigidity}.

\bigskip

\noindent {\bf Acknowledgements.} 
The authors would like to thank Colin Guillarmou for helpful remarks related to Theorem \ref{thm_pestov_carleman_shifted_intro}, and the referee for helpful comments. GPP was partially supported by EPSRC grant EP/R001898/1. MS was partly supported by the Academy of Finland (Finnish Centre of Excellence in Inverse Problems Research, grant numbers 284715 and 309963) and by the European Research Council under FP7/2007-2013 (ERC StG 307023) and Horizon 2020 (ERC CoG 770924).

\section{Geometric preliminaries} \label{sec_preliminaries}


\noindent {\bf Unit sphere bundle.}
To begin, we need to recall certain notions related to the geometry of the unit sphere bundle. We follow the setup and notation of \cite{PSU_hd}; for other approaches and background information see \cite{GK2,Sh,Pa,Kn,DS}.

Let $(M,g)$ be a $d$-dimensional compact Riemannian manifold with or without boundary, having unit sphere bundle $\pi: SM\to M$, and let $X$ be the geodesic vector field. We equip $SM$ with the Sasaki metric. If $\mathcal V$ denotes the vertical subbundle
given by $\mathcal V=\mbox{\rm Ker}\,d\pi$, then there is an orthogonal splitting with respect to the Sasaki metric:
\begin{equation}\label{TSM}
TSM=\re X\oplus {\mathcal H}\oplus {\mathcal V}.
\end{equation}
The subbundle ${\mathcal H}$ is called the horizontal subbundle. Elements in $\mathcal H(x,v)$ and $\mathcal V(x,v)$ are canonically identified with elements in the codimension one subspace $\{v\}^{\perp}\subset T_{x}M$ by the isomorphisms
\[ d\pi_{x,v} : \mc{H}(x,v)\to \{v\}^{\perp} ,  \quad \mc{K}_{x,v}: \mathcal V(x,v)\to \{v\}^{\perp},\]
here $\mc{K}_{(x,v)}$ is the connection map coming from Levi-Civita connection.
We will use these identifications freely below.  

We shall denote by $\mathcal Z$ the set of smooth functions $Z:SM\to TM$ such that $Z(x,v)\in T_{x}M$ and $\langle Z(x,v),v\rangle=0$ for all $(x,v)\in SM$.
Another way to describe the elements of $\mathcal Z$ is a follows. Consider the pull-back bundle $\pi^*TM$ over $SM$.  Let $N$ denote the subbundle of $\pi^*TM$ whose fiber over $(x,v)$
is given by $N_{(x,v)}=\{v\}^{\perp}$ (which we can also identify with $T_{v}S_{x}M$). Then $\mathcal Z$ coincides with the smooth sections
of the bundle $N$. Notice that $N$ carries a natural scalar product and thus an $L^{2}$-inner product 
(using the Liouville measure on $SM$ for integration).

Given a smooth function $u\in C^{\infty}(SM)$ we can consider its gradient $\nabla u$ with respect to the Sasaki metric. 
Using the splitting above we may write uniquely in the decomposition \eqref{TSM}
\[\nabla u=((Xu)X,\hd u,  \vd u). \]
The derivatives $\hd u\in  \mc{Z}$ and $\vd u\in \mc{Z}$ are called horizontal and vertical derivatives respectively. Note that this differs from the definitions in \cite{Kn,Sh} since here all objects are defined on $SM$ as opposed to $TM$.

Observe that $X$ acts on $\mathcal Z$ as follows:
\begin{equation}\label{XonZ}
XZ(x,v):=\frac{DZ(\varphi_{t}(x,v))}{dt}|_{t=0}
\end{equation}
where $D/dt$ is the covariant derivative with respect to Levi-Civita connection and $\varphi_t$ is the geodesic flow.  With respect to the $L^2$-product on $N$, the formal adjoints of $\vd:C^{\infty}(SM)\to\mathcal Z$ and $\hd:C^{\infty}(SM) \to \mathcal Z$ are denoted by $-\vdiv$ and $-\hdiv$ respectively. Note that since $X$ leaves invariant the volume form of the Sasaki metric we have $X^*=-X$ for both actions of $X$ on $C^{\infty}(SM)$ and $\mathcal Z$.
In what follows, we will need to work with the complexified version of $N$ with its natural inherited Hermitian product. This will be clear from the context and we shall employ the same letter $N$ to denote the complexified
bundle and also $\mathcal Z$ for its sections.

Let $R(x,v):\{v\}^{\perp}\to \{v\}^{\perp}$ be the operator determined by the Riemann curvature tensor by $R(x,v)w=R_{x}(w,v)v$, and let $d=\dim(M)$. \\[-5pt]

\noindent {\bf Spherical harmonics decomposition.}
Recall the spherical harmonics decomposition with respect to the vertical Laplacian $\Delta = \vdiv \vd$ (cf.\ \cite[Section 3]{PSU_hd}):
$$
L^2(SM,\C^n) = \bigoplus_{m=0}^{\infty} H_m(SM,\C^n),
$$
so that any $f \in L^2(SM,\C^n)$ has the orthogonal decomposition 
$$
f = \sum_{m=0}^{\infty} f_m.
$$
We write $\Omega_m = H_m(SM,\C^n) \cap C^{\infty}(SM,\C^n)$. Then $-\Delta u = \lambda_m u$ for $u \in \Omega_m$, where we set 
\[
\lambda_{m}:=m(m+d-2).
\]

\noindent {\bf Decomposition of $X$.} The geodesic vector field behaves nicely with respect to the decomposition into fibrewise
spherical harmonics: it maps $\Omega_{m}$ into $\Omega_{m-1}\oplus\Omega_{m+1}$ \cite[Proposition 3.2]{GK2}. Hence on $\Omega_{m}$ we can write
\[X=X_{-}+X_{+}\] 
where $X_{-}:\Omega_{m}\to \Omega_{m-1}$ and $X_{+}:\Omega_{m}\to\Omega_{m+1}$.
By \cite[Proposition 3.7]{GK2} the operator $X_{+}$ is overdetermined elliptic (i.e.\ it has injective principal symbol). One can gain insight into the decomposition $X=X_{-}+X_{+}$ as follows. Fix $x\in M$ and consider local coordinates which are geodesic at $x$ (i.e. all Christoffel symbols vanish at $x$).
Then $Xu(x,v)=v^{i}\frac{\partial u}{\partial x^{i}}$. We now use the following basic fact about spherical harmonics: the product of a spherical harmonic of degree $m$ with a spherical harmonic of degree one decomposes as the sum of spherical harmonics of degree $m-1$ and $m+1$.

\section{Useful analogies} \label{sec_analogies}

\noindent {\bf Analogies with elliptic operators.} It is instructive to compare our approach to X-ray transforms based on energy methods and related well known energy methods for elliptic operators in $\mR^n$. Let $\Omega \subset \mR^n$ be a bounded domain, and let $P = -\Delta + V$ be the Schr\"odinger operator in $\Omega$ where $V \in L^{\infty}(\Omega)$. We consider the uniqueness problem for solutions of the equation 
\begin{equation} \label{intro_schrodinger_problem}
Pu = 0 \text{ in $\Omega$}, \qquad u|_{\partial \Omega} = 0.
\end{equation}

It is well known that under a positivity condition for the potential $V$, any solution $u$ of \eqref{intro_schrodinger_problem} in $C^2(\overline{\Omega})$ (say) must be zero. In fact, if we assume that $V \geq 0$, this follows from the simple energy estimate where we integrate the equation $Pu = 0$ against the test function $u$. This implies that, in terms of $L^2(\Omega)$ inner products, 
\[
0 = (Pu, u) = (-\Delta u, u) + (Vu, u) \geq \norm{\nabla u}^2.
\]
Here we integrated by parts using that $u|_{\partial \Omega} = 0$, and used that $V \geq 0$. Thus $\nabla u = 0$, showing that $u$ is constant, and the boundary condition $u|_{\partial \Omega} = 0$ implies that $u \equiv 0$.

Uniqueness for solutions of \eqref{intro_schrodinger_problem} still holds under the weaker condition $V > -\lambda_1$, where $\lambda_1 > 0$ is the first Dirichlet eigenvalue of $-\Delta$ in $\Omega$. Then the Poincar\'e inequality implies that $\norm{\nabla w}^2 \geq \lambda_1 \norm{w}^2$ whenever $w|_{\partial \Omega} = 0$. The same argument as above shows that 
\[
0 = (Pu, u) = (-\Delta u, u) + (Vu, u) = \norm{\nabla u}^2 + (Vu, u) \geq ( (\lambda_1 + V) u, u).
\]
If $\lambda_1 + V \geq c$ a.e.\ in $\Omega$ for some $c > 0$, it follows that $u \equiv 0$. (Combining this argument with the unique continuation principle, it would be enough to assume that $\lambda_1+V|_U > 0$ in some set $U$ of positive measure.)

If the potential $V$ is very negative, uniqueness for solutions of \eqref{intro_schrodinger_problem} may fail. For example, if $V = -\lambda_1$ where $\lambda_1$ is the first Dirichlet eigenvalue, then the corresponding Dirichlet eigenfunction is a nontrivial solution of \eqref{intro_schrodinger_problem}. However, uniqueness will be true if we assume more vanishing. For instance, any $u \in C^2(\overline{\Omega})$ satisfying 
\begin{equation} \label{intro_schrodinger_problem2}
Pu = 0 \text{ in $\Omega$}, \qquad u|_{\partial \Omega} = \partial_{\nu} u|_{\Omega} = 0
\end{equation}
must be identically zero. This follows from the unique continuation principle, which is typically established by using Carleman estimates (i.e.\ exponentially weighted $L^2$ estimates for $P$).

Carleman estimates themselves correspond to a version of energy methods (more precisely, positive commutator methods), which makes use of the gauge invariance of the problem: writing $u = e^{\varphi} w$ for some $\varphi \in C^2(\overline{\Omega})$, the problem \eqref{intro_schrodinger_problem2} is equivalent with 
\[
P_{\varphi} w = 0 \text{ in $\Omega$}, \qquad w|_{\partial \Omega} = \partial_{\nu} w|_{\Omega} = 0
\]
where $P_{\varphi} = e^{-\varphi} P e^{\varphi}$ is the operator $P$ conjugated with an exponential weight. For some choices of $\varphi$ (for which the commutator $[P_{\varphi}^*, P_{\varphi}]$ is positive), the operator $P_{\varphi}$ is ''more positive'' than $P$, and an energy estimate for $P_{\varphi} w = 0$ (where one integrates against the test function $P_{\varphi} w$) implies that $w \equiv 0$.

Let us now return from the case of the Schr\"odinger operator back to X-ray transforms. As explained in \cite{PSU1, PSU3}, in this case we are considering the operator $P = \vd X$ on $SM$, and we wish to show (say) that any $u \in C^{\infty}(SM)$ satisfying 
\[
Pu = 0 \text{ in $SM$}, \qquad u|_{\partial(SM)} = 0
\]
must be identically zero. The basic energy identity for the operator $P$ is \eqref{eq:pestov}, where
the operator $-X^2 - R$ formally corresponds to the operator $-\Delta + V$ in the above examples. In particular, $-R$ (curvature) plays a similar role as $V$ (potential). For instance, the condition $K \leq 0$ where $K$ denotes sectional curvature corresponds to $V \geq 0$. Moreover, the simple/Anosov condition for $(M,g)$ ensures that $-X^2 - R > 0$ which corresponds to $V > -\lambda_1$ (i.e.\ $-\Delta + V > 0$).

Let us formulate these analogies in the following table:

\vspace{10pt}

{\begin{center}
\begin{tabular}{|c|c|}
\hline
$P = -\Delta + V$ & $P = \vd X$ \\[2pt]
\hline 
potential $V$ & $\phantom{\vd}$curvature $-R$$\phantom{\vd}$ \\[3pt]
$V \geq 0$ & $K \leq 0$ \\[3pt]
$V > -\lambda_1$ & simple/Anosov \\[3pt]
Carleman estimates & \hspace{50pt}?\hspace{50pt} \\[5pt]
\hline
\end{tabular}
\end{center}}

An analogue of the Carleman estimates approach for uniqueness in elliptic equations has so far been missing in the case of X-ray transforms. Theorem \ref{thm_carleman_negativecurvature_intro} represents the first progress in this direction. \\[-5pt]

\noindent {\bf Properties of $P$.} Let us consider in more detail the operator $P = \vd X$ on $SM$ that is fundamental in the energy method for geodesic X-ray transforms. The operator $P$ is a second order differential operator, scalar if $d=2$ and vector-valued if $d \geq 3$, on the compact $(2d-1)$-dimensional manifold $SM$ with boundary. However, even in the case $d=2$ where $P = VX$ is the product of two vector fields, the equation $Pu = 0$ does not seem to fall into any known class of PDEs for which there would be a uniqueness theory.

One major issue is that $P$ is not of principal type. To see this, let $d = 2$ and write $P = P_1 P_2$ where $P_1 = V$ and $P_2 = X$. Then $P_1$ and $P_2$ are smooth vector fields on $SM$. Following \cite[Section 2]{PSU1} (and writing $x_3 = \theta$) we consider local coordinates $(x_1,x_2,x_3)$ on $SM$ so that $g = e^{2\lambda(x_1,x_2)} (dx_1^2 + dx_2^2)$. In this notation, the principal symbols of $P_1$ and $P_2$ are 
\[
p_1 = \xi_3, \qquad p_2 = e^{-\lambda} ( \cos(x_3) \xi_1 + \sin(x_3) \xi_2 + [-\partial_{x_1} \lambda \sin(x_3) + \partial_{x_2} \lambda \cos(x_3)] \xi_3).
\]
Thus $P$ has real principal symbol $p = p_1 p_2$, and the characteristic set $p^{-1}(0)$ is the union of $p_1^{-1}(0)$ and $p_2^{-1}(0)$. However, the intersection 
\[
p_1^{-1}(0) \cap p_2^{-1}(0) = \{ (x, \xi) \in T^*(SM) \,;\, \xi_3 = 0, \ \cos(x_3) \xi_1 + \sin(x_3) \xi_2 = 0 \}
\]
is nontrivial and in this set $dp = p_1 \,dp_2 + p_2 \,dp_1$ vanishes. This means that any null bicharacteristic curve through $p_1^{-1}(0) \cap p_2^{-1}(0)$ reduces to a point, and in particular $P$ is not of principal type (see \cite[Chapter XXVI]{Hormander} for more on principal type operators).

The fact that $P$ is not of principal type means that its properties may depend on lower order terms. Indeed, if $d=2$ one can find first order operators $W$ on $SM$ so that $P + W$ has nontrivial compactly supported solutions. To see this, note that if $W = X_{\perp}$ then $P + W = XV$ (see \cite{PSU1}), and any $u \in C^{\infty}_c(M^{\mathrm{int}})$ satisfies $XVu = 0$ in $SM$. Moreover, the counterexample in \cite{Boman} implies that one may even find such an operator $W$ arising from a weighted X-ray transform: when $(M,g)$ is the Euclidean unit disc, there is $\mathcal{A} \in C^{\infty}(SM)$ such that $V(X+\mathcal{A}) u = 0$ for some nontrivial $u \in C^{\infty}_c(SM^{\mathrm{int}})$. These observations indicate that the structure of lower order terms is crucial for the uniqueness problem, and principal symbol computations will not be sufficient.

However, in spite of the above issues it is possible to obtain uniqueness results for $P$ via energy methods. Based on the Pestov energy identity, in \cite{PSU_hd} it was proved that if $(M,g)$ is a compact simple/Anosov manifold, one has the inequality 
\[
\norm{u - (u)_{SM}}_{H^1(SM)} \leq C \norm{Pu}_{L^2(SM)}
\]
valid for all $u \in H^2(SM)$ (with $u|_{\partial(SM)} = 0$ in the boundary case). Here $(u)_{SM} = \frac{1}{\text{Vol}(SM)} \int_{SM} u$. Such an inequality immediately implies that if $Pu = 0$ in $SM$ and $u|_{\partial(SM)} = 0$, then $u \equiv 0$.

Various generalizations of the geodesic X-ray transform lead to estimates of the form 
\begin{equation} \label{pestov_estimate_error_term}
\norm{u - (u)_{SM}}_{H^1(SM)} \leq C \norm{Pu}_{L^2(SM)} + \text{error term}
\end{equation}
where, roughly, 
\[
\begin{array}{ll}
\left\{ \begin{array}{c} \text{unitary connections} \\ \text{higher order tensors} \end{array} \right. \rightsquigarrow \text{$H^{1/2}$ error}, \\[15pt]
\left\{ \begin{array}{c} \text{general connections} \\ \text{curvature} \\ \text{spatial localization} \end{array} \right. \rightsquigarrow \text{$H^1$ error}. 
\end{array}
\]
To justify these statements, note that a unitary connection $A$ introduces the term $-(F_A u, \vd u)$ in the Pestov identity (cf.\ Lemma \ref{lemma_pestov_general_connection} below), and that when $d=2$ the argument leading to \cite[equation (6)]{PSU1} shows that higher order tensors introduce a term of the form $\sum_{k=-\infty}^{\infty} |k| (K v_k, v_k)$ (this extends to any dimension), where $K$ is the Gaussian curvature and the expansion is taken with respect to the vertical Fourier decomposition of functions in the unit circle bundle. Both cases correspond to $H^{1/2}$ error terms in \eqref{pestov_estimate_error_term}. However, a general (non-unitary) connection $A$ can be introduced to the Pestov identity by writing $P$ as $P + \vd A - \vd A$ and using the triangle inequality, leading to an $H^1$ error term $\norm{\vd(Au)}$. The curvature term $-(R \vd u, \vd u)$ in the Pestov identity is also of $H^1$ type, and spatial localization corresponds to commuting $P$ with a spatial cutoff function which again introduces an $H^1$ error.

The above statements indicate that unitary connections and higher order tensors lead to $H^{1/2}$ errors in the energy estimate \eqref{pestov_estimate_error_term}. This is a weaker norm than the $H^1$ norm appearing on the left hand side of \eqref{pestov_estimate_error_term}, which suggests that such perturbations might be manageable (indeed, the high frequency components in the error term can be absorbed, which in particular implies that there is a finite dimensional kernel). However, general connections, curvature or spatial localization correspond to $H^1$ error terms which are as strong as the $H^1$ term on the left, suggesting that these perturbations may be challenging to handle by energy methods (it seems that so far only the method in \cite{UhlmannVasy} can really deal with such perturbations). In this article, we introduce energy methods that are able to deal with $H^1$ perturbations arising from general connections. \\[-5pt]

\section{Pestov identity with general connection} \label{section_pestov_general_connection}

In this section we will prove a version of the Pestov identity that involves a general connection, extending the version for unitary connections proved in \cite{GPSU}. We also give an equivalent version stated in terms of the $X_{\pm}$ operators.

We emphasize that for the applications we will eventually only use the identity for $A=0$. However, the general setup here highlights the fact that dealing with general connections results in a lack of symmetry in the energy estimates. This lack of symmetry indicates that general connections are indeed stronger perturbations than unitary connections, as explained in Section \ref{sec_analogies}, and it forces one to use other methods such as the Carleman estimates in this article. For simplicity of presentation we will work on the trivial bundle $M \times \mC^n$, but straightforward modifications would lead to analogous results on general Hermitian bundles as in \cite{GPSU}. We will use the notation from Section \ref{sec_preliminaries}.

Let $(M,g)$ be a compact manifold with or without boundary, with $d = \dim(M)$. Let $A$ be an $n \times n$ matrix of smooth complex $1$-forms on $M$, or equivalently a smooth function $A: SM \to \mC^{n \times n}$ so that $A(x,v)$ is linear in $v$. Then $A$ defines a connection on the trivial bundle $M \times \mC^n$. We define the following operators on $C^{\infty}(SM, \mC^n)$, 
\begin{align*}
X^A &:= X + A, \\
\hd_A &:= \hd + (\vd A).
\end{align*}
Here $A$ and $(\vd A)$ act by multiplication. The horizontal divergence $\hdiv_A$ is defined for $Z \in \mathcal{Z}^n$ by 
\[
\hdiv_A \,Z := \hdiv \,Z + \langle \vd A, Z \rangle.
\]
Finally, we define $F_A\in C^{\infty}(SM,N\otimes \mC^{n\times n})$ by 
\[
F_A := X(\vd A) - \hd A + [A, \vd A].
\]
The element $F_{A}$ acts on functions $u\in C^{\infty}(SM,\C^n)$ by matrix multiplication and thus it induces and operator of order zero.
In the natural $L^2$ inner product, one has 
\[
(X^A)^* = -X^{-A^*}, \qquad (\hd_A)^* = -\hdiv_{-A^*}.
\]
Note that if $A$ is a unitary connection, meaning that $A^* = -A$, then of course $(X^A)^* = -X^A$, $(\hd_A)^* = -\hdiv_{A}$, and $F_A$ is the curvature operator in \cite{GPSU}.

With the above conventions, one can check by direct computations that the basic commutator formulas given in \cite{PSU_hd} for $A=0$ and in \cite{GPSU} for unitary $A$ remain true for a general connection $A$.
Recall that $R$ is the curvature operator defined before and acts on $C^{\infty}(SM,N\otimes \C^n)$ as $R\id_{n\times n}$.

\begin{Lemma} \label{lemma_basic_commutator_identities}
One has the following commutator identities on $C^{\infty}(SM,\mC^n)$, 
\begin{align*}
[X^A, \vd] &= -\hd_A, \\
[X^A, \hd_A] &= R \vd + F_A, \\
\hdiv_A \vd - \vdiv \hd_A &= (d-1) X^A,
\end{align*}
and by duality one gets the following identity on $\mathcal{Z}^n$,
\[
[X^A, \vdiv] = -\hdiv_A.
\]
\end{Lemma}

The paper \cite{GPSU} gave a version of the Pestov identity with unitary connection in any dimension. If the connection is not unitary, one loses symmetry in the Pestov identity but the following form of this identity remains valid.

\begin{Lemma} \label{lemma_pestov_general_connection}
If $A$ is a general connection, one has the identity 
\begin{multline*}
(\vd X^A u, \vd X^{-A^*} u) \\
 = (X^A \vd u, X^{-A^*} \vd u) - (R \vd u, \vd u) - (F_A u, \vd u) + (d-1) (X^A u, X^{-A^*} u)
\end{multline*}
for any $u \in C^{\infty}(SM, \mC^n)$ with $u|_{\partial(SM)} = 0$ in the boundary case.
\end{Lemma}
\begin{proof}
Observe that 
\[
(\vd X^A u, \vd X^{-A^*} u) - (X^A \vd u, X^{-A^*} \vd u) = ( (X^A \vdiv \vd X^A - \vdiv X^A X^A \vd) u, u).
\]
The commutator identities above imply that 
\[
X^A \vdiv \vd X^A - \vdiv X^A X^A \vd = -(d-1)X^A X^A + \vdiv (R \vd + F_A).
\]
The result follows.
\end{proof}

If $A$ is a general connection, the same argument as in Section \ref{sec_preliminaries} shows that $X^{A}$ maps $\Omega_m$ to $\Omega_{m-1} \oplus \Omega_{m+1}$. Thus we have the decomposition 
\[
X^A = X^A_+ + X^A_-, \qquad X^A_{\pm}: \Omega_m \to \Omega_{m \pm 1}.
\]
We will next rewrite the Pestov identity in terms of $X^A_+$ and $X^A_-$. To do this, we need some notation: for a polynomially bounded sequence $\alpha = (\alpha_\l)_{\l=0}^{\infty}$ of real numbers, we define a corresponding ``inner product'' 
\[
(u,w)_{\alpha} = \sum_{\l=0}^{\infty} \alpha_\l (u_\l, w_\l)_{L^2(SM)}, \qquad u, w \in C^{\infty}(SM, \mC^n).
\]
We also write $\norm{u}_\alpha^2 = \sum_{\l=0}^{\infty} \alpha_\l \norm{u_\l}^2$. (If each $\alpha_\l$ is positive one gets an actual inner product and norm, but it is notationally convenient to allow zero or negative $\alpha_\l$.)

The Pestov identity can then be written in the following form.

\begin{Lemma} \label{lemma_pestov_xplusminus_generalconnection}
If $u \in C^{\infty}(SM,\C^n)$ with $u|_{\partial(SM)} = 0$ in the boundary case, then 
\[
(X^A_- u, X^{-A^*}_- u)_{\alpha} - (R \vd u, \vd u) - (F_A u, \vd u) + (Z_A(u), Z_{-A^*}(u)) = (X^A_+ u, X^{-A^*}_+ u)_{\beta}
\]
where $Z_A(u)$ is the $\vdiv$-free part of $\hd_A u$, and 
\begin{align*}
\alpha_\l &= \left\{ \begin{array}{cl} d-1, & \l = 0, \\ (2\l+d-2) \left(1+\frac{1}{\l+d-2} \right), & \l \geq 1, \end{array} \right. \\
\beta_\l &= \left\{ \begin{array}{cl} 0, & \l=0,1, \\ (2\l+d-2) \left(1-\frac{1}{\l} \right), & \l \geq 2. \end{array} \right.
\end{align*}
\end{Lemma}
\begin{proof} Note that $u\in C^{\infty}(SM,\C^n)$, implies rapid decay, namely, $u_{\l}=O((1+|\l|)^{-\infty})$.
We will use Lemma \ref{lemma_pestov_general_connection} in the form 
\[
(\vd X^A u, \vd X^{-A^*} u) - (d-1) (X^A u, X^{-A^*} u) = (X^A \vd u, X^{-A^*} \vd u) - (R \vd u, \vd u) - (F_A u, \vd u).
\]
Note that Lemma \ref{lemma_basic_commutator_identities} gives
\begin{equation} \label{xavdu_decomposition}
X^A \vd u = \vd X^A u - \hd_A u.
\end{equation}
We also have the commutator formula $[X^A, -\vdiv \vd] = 2 \vdiv \hd_A + (d-1)X^A$ by the identities in Lemma \ref{lemma_basic_commutator_identities} (compare with \cite[Lemma 3.5]{PSU_hd}). Thus we obtain as in \cite[Lemma 4.4]{PSU_hd} that 
\begin{equation} \label{hdau_decomposition}
\hd_A u = \vd \left[ \sum_{\l=1}^{\infty} \left( \frac{1}{\l} X^A_+ u_{\l-1} - \frac{1}{\l+d-2} X^A_- u_{\l+1} \right) \right] + Z_A(u)
\end{equation}
where $Z_A(u) \in \mathcal Z^n$ satisfies $\vdiv \,Z_A(u) = 0$. Thus \eqref{xavdu_decomposition} and \eqref{hdau_decomposition} yield that 
\[
X^A \vd u = \vd \sum_{\l=1}^{\infty} \left[ \left(1-\frac{1}{\l} \right) X^A_+ u_{\l-1} + \left(1+\frac{1}{\l+d-2} \right) X^A_- u_{\l+1} \right] - Z_A(u).
\]
Applying this for $A$ and $-A^*$ gives (recall that $\lambda_{\l}=\l(\l+d-2)$)
{\scriptsize 
\begin{align*}
 &(X^A \vd u, X^{-A^*} \vd u)  \\
 & = \sum_{\l=1}^{\infty} \lambda_\l \left( \left(1-\frac{1}{\l} \right) X^A_+ u_{\l-1} + \left(1+\frac{1}{\l+d-2} \right) X^A_- u_{\l+1},  \left(1-\frac{1}{\l} \right) X^{-A^*}_+ u_{\l-1} + \left(1+\frac{1}{\l+d-2} \right) X^{-A^*}_- u_{\l+1} \right) \\
 &\qquad + (Z_A(u), Z_{-A^*}(u)) \\
 &= \sum_{\l=1}^{\infty} \lambda_\l \left[ \left(1-\frac{1}{\l} \right)^2 (X^A_+ u_{\l-1}, X^{-A^*}_+ u_{\l-1}) + \left(1+\frac{1}{\l+d-2} \right)^2 (X^A_- u_{\l+1}, X^{-A^*}_- u_{\l+1}) \right] \\
 & \qquad + \sum_{\l=1}^{\infty} \lambda_\l \left(1-\frac{1}{\l} \right) \left(1+\frac{1}{\l+d-2} \right) \left[ (X^A_+ u_{\l-1}, X^{-A^*}_- u_{\l+1}) + (X^A_- u_{\l+1}, X^{-A^*}_+ u_{\l-1}) \right] + (Z_A(u), Z_{-A^*}(u)).
\end{align*}
}
On the other hand, one has 
\begin{align*}
 &(\vd X^A u, \vd X^{-A^*} u) - (d-1) (X^A u, X^{-A^*} u) \\
 &= -(d-1) (X^A_- u_1, X^{-A^*}_- u_1) + \sum_{\l=1}^{\infty} (\lambda_l - (d-1)) (X^A_+ u_{\l-1} + X^A_- u_{\l+1}, X^{-A^*}_+ u_{\l-1} + X^{-A^*}_- u_{\l+1}) \\
 &= -(d-1) (X^A_- u_1, X^{-A^*}_- u_1) + \sum_{\l=1}^{\infty} (\lambda_\l - (d-1)) \left[ (X^A_+ u_{\l-1}, X^{-A^*}_+ u_{\l-1} ) + (X^A_- u_{\l+1}, X^{-A^*}_- u_{\l+1}) \right] \\
 &\qquad + \sum_{\l=1}^{\infty} (\lambda_\l - (d-1)) \left[ (X^A_+ u_{\l-1}, X^{-A^*}_- u_{\l+1}) + (X^A_- u_{\l+1}, X^{-A^*}_+ u_{\l-1}) \right].
\end{align*}
Somewhat miraculously, we observe that 
\[
\lambda_\l \left(1-\frac{1}{\l} \right) \left(1+\frac{1}{\l+d-2} \right) = \lambda_\l - (d-1).
\]
This means that the two sums above involving $\left[ (X^A_+ u_{\l-1}, X^{-A^*}_- u_{\l+1}) + (X^A_- u_{\l+1}, X^{-A^*}_+ u_{\l-1}) \right]$ terms are equal. The Pestov identity in the beginning of the proof now yields 
\begin{multline*}
-(d-1) (X^A_- u_1, X^{-A^*}_- u_1) + \sum_{\l=1}^{\infty} (\lambda_\l - (d-1)) \left[ (X^A_+ u_{\l-1}, X^{-A^*}_+ u_{\l-1} ) + (X^A_- u_{\l+1}, X^{-A^*}_- u_{\l+1}) \right]  \\
 = \sum_{\l=1}^{\infty} \lambda_\l \left[ \left(1-\frac{1}{\l} \right)^2 (X^A_+ u_{\l-1}, X^{-A^*}_+ u_{\l-1}) + \left(1+\frac{1}{\l+d-2} \right)^2 (X^A_- u_{\l+1}, X^{-A^*}_- u_{\l+1}) \right] \\ - (R \vd u, \vd u) - (F_A u, \vd u) + (Z_A(u), Z_{-A^*}(u)).
\end{multline*}
We rewrite this as 
\begin{multline*}
\sum_{\l=0}^{\infty} \alpha_\l (X^A_- u_{\l+1}, X^{-A^*}_- u_{\l+1}) - (R \vd u, \vd u) - (F_A u, \vd u) + (Z_A(u), Z_{-A^*}(u)) \\
 = \sum_{\l=1}^{\infty} \beta_\l (X^A_+ u_{\l-1}, X^{-A^*}_+ u_{\l-1})
\end{multline*}
where 
\begin{align*}
\alpha_\l &= \lambda_\l \left[ \left(1+\frac{1}{\l+d-2} \right)^2 - 1 \right] + (d-1), \\
\beta_\l &= \lambda_\l \left[ 1 -  \left(1-\frac{1}{\l} \right)^2 \right]  - (d-1).
\end{align*}
The result follows after simplifying the expressions for $\alpha_\l$ and $\beta_\l$.
\end{proof}

We will use the previous identity only in the case where the connection $A$ is unitary, or when $A=0$. In these cases the inner products become squares of $L^2$ norms, and we obtain the following energy identity which is equivalent with the Pestov identity with unitary connection given in \cite{GPSU}.

\begin{Proposition}[Pestov identity in terms of $X^A_{\pm}$] \label{prop_pestov_xplus_xminus}
Let $(M,g)$ be a compact manifold with or without boundary, and let $A$ be a unitary connection. Then 
\[
\norm{X^A_- u}_{\alpha}^2 - (R \vd u, \vd u) - (F_A u, \vd u) + \norm{Z_A(u)}^2 = \norm{X^A_+ u}_{\beta}^2
\]
for any $u \in C^{\infty}(SM, \mC^n)$ with $u|_{\partial(SM)} = 0$ in the boundary case.
\end{Proposition}

\section{Frequency localization} \label{sec_frequency_localization}

Recall that any $u \in C^{\infty}(SM, \mC^n)$ admits an $L^2$-orthogonal decomposition 
\[
u = \sum_{\l=0}^{\infty} u_\l, \qquad u_\l \in \Omega_\l,
\]
where $\Omega_\l$ corresponds to the set of vertical spherical harmonics of degree $\l$. Since $X_{\pm}^A$ maps $\Omega_\l$ to $\Omega_{\l \pm 1}$, it is immediate that the Pestov identity with unitary connection (Proposition \ref{prop_pestov_xplus_xminus}) reduces to the following identity when applied to functions in $\Omega_\l$.

\begin{Proposition}[Pestov identity on $\Omega_\l$] \label{prop_pestov_localized}
Let $(M,g)$ be a compact manifold with or without boundary, let $A$ be a unitary connection, and let $\l \geq 0$. One has 
\[
\alpha_{\l-1} \norm{X^A_- u}^2 - (R \vd u, \vd u) - (F_A u, \vd u) + \norm{Z_A(u)}^2 = \beta_{\l+1} \norm{X^A_+ u}^2, \qquad u \in \Omega_\l,
\]
if additionally $u|_{\partial(SM)} = 0$ in the boundary case. (We define $\alpha_{-1} = 0$.)
\end{Proposition}

If $\dim(M) = 2$, the Pestov identity on $\Omega_\l$ is the same as the Guillemin-Kazhdan energy identity \cite{GK}. In \cite[Appendix B]{PSU_hd} it was observed that in two dimensions the Guillemin-Kazhdan identity is actually equivalent with the Pestov identity, in the sense that summing that Guillemin-Kazhdan identity over all $\l$ gives back the Pestov identity.

We will now show that the same is true in any dimension: the Pestov identity is equivalent with the frequency localized identities of Proposition \ref{prop_pestov_localized}. This means that the Pestov identity localizes completely with respect to vertical Fourier decompositions. This will be a very important observation in what follows.

\begin{Lemma} \label{lemma_pestov_equivalence}
The Pestov identity on $\Omega_\l$ is equivalent with the Pestov identity with unitary connection in the following sense: for any $u \in C^{\infty}(SM, \mC^n)$ with $u|_{\partial(SM)} = 0$ in the boundary case, one has 
\begin{multline*}
\sum_{\l=0}^{\infty} \left[ \alpha_{\l-1} \norm{X^A_- u_\l}^2 - (R \vd u_\l, \vd u_\l) - (F_A u_\l, \vd u_\l) + \norm{Z_A(u_\l)}^2 -\beta_{\l+1} \norm{X^A_+ u_\l}^2 \right] \\
 = \norm{X^A_- u}_{\alpha}^2 - (R \vd u, \vd u) - (F_A u, \vd u) + \norm{Z_A(u)}^2 - \norm{X^A_+ u}_{\beta}^2.
\end{multline*}
\end{Lemma}

The result will follow if we can show that the curvature and $Z_A$ terms localise. Thus Lemma \ref{lemma_pestov_equivalence} is a corollary of the next result.

\begin{Lemma} \label{lemma_curvature_fourier_localization}
If $(M,g)$ is a Riemannian manifold and $A$ is a unitary connection, then 
\[
(R\vd u, \vd w) = 0, \qquad (F_A u, \vd w) = 0, \qquad (Z_A(u), Z_A(w)) = 0,
\]
whenever $u \in \Omega_m$, $w \in \Omega_\l$ and $m \neq \l$.
\end{Lemma}

To prove this lemma, we first prove two auxiliary lemmas. For the auxiliary lemmas we work in $\re^n$ with the standard inner product and let $S^{n-1}$ be the unit sphere. We denote by $\Omega_m\subset C^{\infty}(S^{n-1})$ the space of spherical harmonics of degree $m$; let $\nabla$ denote the gradient in $S^{n-1}$ with the canonical metric induced by $\re^n$.

\begin{Lemma} Let $\alpha$ be any anti-symmetric 2-form in $\re^n$. Given $u\in \Omega_m$, the function
\[S^{n-1}\ni x\mapsto \alpha(x,\nabla u(x))\]
belongs to $\Omega_m$.
\label{lemma:anti}

\end{Lemma}

\begin{proof} It suffices to prove the claim for elements $dx_{i}\wedge dx_{j}$ of the standard basis of $\Lambda^{2}(\re^{n})^*$.
Let $f:\re^{n}\setminus\{0\}\to S^{n-1}$ be the map $f(x)=x/|x|$ and introduce the vector fields on $S^{n-1}$ given by
\[Y_{j}(x):=df_{x}(\partial /\partial x_{j}),\;\;\;\text{for}\;x\in S^{n-1}.\]
Since $df_{x}(v)=v-\langle x,v\rangle x$ for $x\in S^{n-1}$ and $v\in\re^{n}$ we have 
\begin{equation}
Y_{j}(x)=\frac{\partial}{\partial x_{j}}-x_{j}x.
\label{eq:Ys}
\end{equation}
One easily checks that
\[\nabla u=Y_{j}(u)\frac{\partial}{\partial x_{j}}=\nabla_{\re^{n}}(f^*u)|_{S^{n-1}}.\]
Using \eqref{eq:Ys} we see that
\[x_{i}Y_{j}-x_{j}Y_{i}=x_{i}\frac{\partial}{\partial x_{j}}-x_{j}\frac{\partial}{\partial x_{i}}\]
and since the latter is a Killing field of $S^{n-1}$ we have
\[[\Delta,x_{i}Y_{j}-x_{j}Y_{i}]=0.\]
Thus
\[dx_{i}\wedge dx_{j}(x,\nabla u)=x_{i}Y_{j}(u)-x_{j}Y_{i}(u)\in \Omega_{m}\]
as desired.
\end{proof}

Given $\alpha,\beta\in \Lambda^{2}((\re^{n})^*)$ recall that the symmetric product of $\alpha$ and $\beta$ as a 4-tensor is given by
\[(\alpha \odot \beta)(x,y,z,w)=\alpha(x,y)\beta(z,w)+\alpha(z,w)\beta(x,y).\]
This gives $\alpha\odot\beta\in S^{2}(\Lambda^{2}(\re^{n})^*)$ and moreover, elements of this form span $S^{2}(\Lambda^{2}(\re^{n})^*)$.

\begin{Lemma}Let $R\in S^{2}(\Lambda^{2}(\re^{n})^*)$. Then
\[\int_{S^{n-1}}R(\nabla u,x,x,\nabla w)\,dx=0\]
for any $u\in\Omega_m$ and $w\in\Omega_{k}$ with $m\neq k$.
\label{lemma:curR}
\end{Lemma}

\begin{proof} Given that elements of the form $\alpha\odot\beta$ span $S^{2}(\Lambda^{2}(\re^{n})^*)$ it suffices to show the claim for such 4-tensors. Using the definition of the symmetric product it is enough to prove that
\[\int_{S^{n-1}}\alpha(\nabla u,x)\beta(x,\nabla w)\,dx=0\]
for any $u\in\Omega_m$ and $w\in\Omega_{k}$ with $m\neq k$. This follows right away from the previous lemma.
\end{proof}


\begin{proof}[Proof of Lemma \ref{lemma_curvature_fourier_localization}]

The proof reduces to a statement for fixed $x$ and the first claim involving the Riemann curvature tensor follows directly from Lemma \ref{lemma:curR}.

For the result involving $F_A$, we first use \cite[Lemma 3.1]{GPSU} and observe that given $u = (u_1, \ldots, u_n) \in \mC^n$ and $Z = (Z_1, \ldots, Z_n)$ with $\langle Z_{\gamma}, v \rangle = 0$ for $1 \leq \gamma \leq n$, one has that the pointwise inner product is given by
\[
\langle F_A u, Z \rangle = \sum_{\gamma,\delta} \langle (F_A)_{\gamma\delta} u_\gamma, Z_\delta \rangle = \sum_{\gamma,\delta} (f_A)_{\gamma\delta}(v, u
_\gamma Z_\delta)
\]
where $f_A = dA + A \wedge A$ is an $n \times n$ matrix of complex $2$-forms. It follows that 
\[
(F_A u, \vd w) = \int_{S_x M} (f_A)_{\gamma \delta}(v, u_\gamma \vd \bar{w}_\delta) \,d(S_x M).
\]
Since $(f_{A})_{\gamma\delta}$ is a 2-form for all $\gamma$ and $\delta$ we may use
Lemma \ref{lemma:anti} to deduce that $(F_A u, \vd w) = 0$ whenever $m \neq \l$.

Finally, to prove to statement involving $Z_A$ we argue as follows.
The definition of $F_A$ and Lemma \ref{lemma_basic_commutator_identities} imply 
\[
\vdiv(F_A) = 0.
\]
Indeed, the bracket relations in Lemma \ref{lemma_basic_commutator_identities} give $\vdiv(X(\vd A) - \hd A)=0$. In order to check that also $\vdiv([A, \vd A])=0$ it suffices to note that if $a,b$ are scalar 1-forms, then
$\vdiv (a\vd b)=\vdiv (b\vd a)$ since $\Delta a=-(d-1)a$ and $\Delta b=-(d-1)b$.
Thus invoking the fact that $A$ is unitary
\begin{equation} \label{fa_symmetry}
(F_A u, \vd w) = \overline{(F_A w, \vd u)}.
\end{equation}
Now observe that Proposition \ref{prop_pestov_xplus_xminus}, the formula \eqref{fa_symmetry}, and the polarization identity imply that 
\[
(Z_A(u), Z_A(w)) = (X^A_+ u, X^A_+ w)_{\beta} - (X^A_- u, X^A_- w)_{\alpha} + (R \vd u, \vd w) + (F_A u, \vd w).
\]
The statements proved above imply that $(Z_A(u), Z_A(w)) = 0$ when $m \neq \l$.
\end{proof}

\section{Carleman estimates in negative curvature} \label{sec_carleman_negative_curvature}

We will now give the first Carleman estimate for X-ray transforms, valid for negative sectional curvature. Throughout this section we will assume that $(M,g)$ is a compact Riemannian manifold with or without boundary, with $d = \dim(M) \geq 2$. The following theorem is the main result of this section.

\begin{Theorem} \label{thm_pestov_carleman_negativecurvature}
Assume that that $(M,g)$ has sectional curvature $\leq -\kappa$ where $\kappa > 0$. For any $s > -1/2$, for any $m \geq 1$ and for any $u \in C^{\infty}(SM, \mC^n)$ (with $u|_{\partial(SM)} = 0$ in the boundary case) one has 
\begin{align*}
 & \sum_{\l=m}^{m+1} \l^{2s+1} \norm{X_- u_\l}^2 + (2s+1) \sum_{\l=m+2}^{\infty} (\l-1)^{2s} \norm{X_- u_\l}^2+ \kappa \sum_{\l=m}^{\infty} \l^{2s+2} \norm{u_\l}^2 \\
 &\qquad \qquad + \sum_{\l=m}^{\infty} \l^{2s} \norm{Z(u_\l)}^2 \leq \frac{C}{2s+1} \sum_{\l=m+1}^{\infty} \l^{2s+2} \norm{(Xu)_{\l}}^2
\end{align*}
where $C = C(d) > 0$.
\end{Theorem}

If $s$ is large, the previous result easily implies the Carleman estimate stated in the introduction.

\begin{Theorem} \label{thm_pestov_carleman_negativecurvature_tau}
Assume that that $(M,g)$ has sectional curvature $\leq -\kappa$ where $\kappa > 0$. For any $\tau \geq 1$ and $m \geq 1$, one has 
\begin{align*}
\sum_{\l=m}^{\infty} e^{2\tau \log(\l)} \norm{u_{\l}}^2 \leq \frac{C}{\kappa \tau} \sum_{\l=m+1}^{\infty} e^{2\tau \log(\l)}\norm{(Xu)_{\l}}^2
\end{align*}
whenever $u \in C^{\infty}(SM, \mC^n)$ (with $u|_{\partial(SM)} = 0$ in the boundary case).
\end{Theorem}
\begin{proof}
It is enough to choose $\tau = s+1$ with $s \geq 0$ in Theorem \ref{thm_pestov_carleman_negativecurvature}.
\end{proof}

As discussed in the introduction, the Carleman estimate in Theorem \ref{thm_pestov_carleman_negativecurvature} can also be understood as a version of the Pestov identity that has been shifted to a different regularity scale. To explain this, introduce the mixed norms 
\begin{align*}
\norm{u}_{L^2_x H^s_v}^2 &= \sum_{\l=0}^{\infty} \br{\l}^{2s} \norm{u_{\l}}^2, \\
\norm{\vd u}_{L^2_x H^s_v}^2 &= \sum_{\l=0}^{\infty} \br{\l}^{2s} \norm{\vd u_{\l}}^2, \\
\norm{\nabla_{SM} u}_{L^2_x H^s_v}^2 &= \sum_{\l=0}^{\infty} \br{\l}^{2s} (\norm{X_- u_{\l+1}}^2 + \norm{X_+ u_{\l-1}}^2 + \norm{Z(u_{\l})}^2 + \norm{\vd u_{\l}}^2)
\end{align*}
where $\br{\l} = (1+\l^2)^{1/2}$, $\norm{\vd u_{\l}}^2 = \lambda_{\l} \norm{u_{\l}}^2$, and $u_{-1} = 0$.

Clearly $\norm{u}_{L^2_x H^0_v} = \norm{u}_{L^2(SM)}$ and $\norm{\vd u}_{L^2_x H^0_v} = \norm{\vd u}_{L^2(SM)}$. We also have that $\norm{\nabla_{SM} u}_{L^2_x H^0_v} \sim \norm{\nabla_{SM} u}_{L^2(SM)}$ by the formula (see \cite[Lemma 5.1]{LRS}) 
\[
\norm{\nabla_{SM} u}^2 = \norm{Xu}^2 + \norm{\hd u}^2 + \norm{\vd u}^2 \sim \norm{X_- u}^2 + \norm{X_+ u}^2 + \norm{Z(u)}^2 + \norm{\vd u}^2,
\]
and using that $\norm{Z(u)}^2 = \sum_{\l=0}^{\infty} \norm{Z(u_{\l})}^2$ by Lemma \ref{lemma_curvature_fourier_localization}.

Theorem \ref{thm_pestov_carleman_negativecurvature} can now be restated as a shifted Pestov identity.

\begin{Theorem} \label{thm_pestov_carleman_shifted}
Assume that that $(M,g)$ has sectional curvature $\leq -\kappa$ where $\kappa > 0$. For any $s > -1/2$ there is $C = C_{d,s,\kappa} > 0$ such that 
\[
\norm{\nabla_{SM} u}_{L^2_x H^s_v} \leq C \norm{\vd Xu}_{L^2_x H^s_v}
\]
for any $u \in C^{\infty}(SM, \mC^n)$ (with $u|_{\partial(SM)} = 0$ in the boundary case).
\end{Theorem}
\begin{proof}
Theorem \ref{thm_pestov_carleman_negativecurvature} with the choice $m=1$ yields that 
\begin{align*}
\norm{X_- u_1}^2 + \norm{X_- u_2}^2 + (2s+1) \sum_{\l=2}^{\infty} \l^{2s} \norm{(X_- u)_{\l}}^2 + \kappa \sum_{\l=1}^{\infty} \l^{2s+2} \norm{u_{\l}}^2 \\
 + \sum_{\l=1}^{\infty} \l^{2s} \norm{Z(u_{\l})}^2 \leq \frac{C}{2s+1} \sum_{\l=2}^{\infty} \l^{2s+2} \norm{(Xu)_{\l}}^2 \leq \frac{C}{2s+1} \norm{\vd Xu}_{L^2_x H^s_v}^2.
\end{align*}
Note also that $Z(u_0) = 0$, which follows from Proposition \ref{prop_pestov_localized} with $\l=0$ (recall that $\beta_1 = 0$). Thus we have 
\[
 \sum_{\l=0}^{\infty} \br{\l}^{2s} (\norm{X_- u_{\l+1}}^2 + \norm{Z(u_{\l})}^2 + \norm{\vd u_{\l}}^2) \leq \frac{C}{2s+1} \norm{\vd Xu}_{L^2_x H^s_v}^2.
\]
Finally, note that 
\[
\norm{X_+ u}_{L^2_x H^s_v}^2 = \sum_{\l=1}^{\infty} \br{\l}^{2s} \norm{(Xu)_{\l} - X_- u_{\l+1}}^2 \leq C(\norm{\vd Xu}_{L^2_x H^s_v}^2 + \norm{X_- u}_{L^2_x H^s_v}^2).
\]
The result follows upon combining the last two inequalities.
\end{proof}

We now begin the proof of Theorem \ref{thm_pestov_carleman_negativecurvature}. The first step is to observe that the localized Pestov identity in Proposition \ref{prop_pestov_localized} gains a positive term in negative sectional curvature.

\begin{Lemma} \label{lemma_pestov_localized_negativecurvature}
Assume that $(M,g)$ has sectional curvature $\leq -\kappa$ where $\kappa > 0$, and let $\l \geq 0$. One has 
\[
\alpha_{\l-1} \norm{X_- u}^2 + \kappa \lambda_{\l} \norm{u}^2 + \norm{Z(u)}^2 \leq \beta_{\l+1} \norm{X_+ u}^2, \qquad u \in \Omega_{\l},
\]
if additionally $u|_{\partial(SM)} = 0$ in the boundary case. (We define $\alpha_{-1} = 0$.)
\end{Lemma}
\begin{proof}
This follows from Proposition \ref{prop_pestov_localized} upon choosing $A = 0$ and noting that 
\[
-(R \vd u, \vd u) \geq \kappa \norm{\vd u}^2 = \kappa (-\vdiv \vd u, u) = \kappa \lambda_{\l} \norm{u}^2. \qedhere
\]
\end{proof}

\begin{Remark}
The $\norm{Z(u)}^2$ term can be simplified when $d=2$. In this case, one has in the notation of \cite[Appendix B]{PSU_hd}
\[
Z(u) = -(X_{\perp} u)_0 iv = (i(\eta_+ u_{-1} - \eta_- u_1)) iv,
\]
which shows in particular that $Z(u_k) = 0$ unless $k=\pm 1$. To prove the above claim, assume that $d=2$ and note that, in the notation of \cite[Appendix B]{PSU_hd}, 
\[
\hd u = -(X_{\perp} u) iv, \qquad \vd a = (Va) iv.
\]
We may write $X_{\perp} u$ as 
\[
X_{\perp} u = (X_{\perp} u)_0 + Va, \qquad a = \sum_{k \neq 0} \frac{1}{ik} (X_{\perp} u)_k.
\]
It follows that 
\[
\hd u = \vd (-a) - (X_{\perp} u)_0 iv.
\]
Since $Z(u)$ was defined as the $\vdiv$-free part of $\hd u$, we must have $Z(u) = - (X_{\perp} u)_0 iv$.
\end{Remark}

The Carleman estimate in Theorem \ref{thm_pestov_carleman_negativecurvature} will follow after multiplying the localized estimates by suitable weights and adding them together. We first give an estimate with rather general weights. In what follows, $C^{\infty}_{F}(SM)$ denotes the set of smooth functions in $SM$ with finite degree.

\begin{Proposition} \label{prop_carleman_general_weight}
Assume that $(M,g)$ has sectional curvature $\leq -\kappa$ where $\kappa > 0$. If $m \geq 1$ and if $(\gamma_{\l})_{\l=0}^{\infty}$ is any sequence of positive real numbers satisfying 
\begin{equation} \label{gammal_condition}
\alpha_{\l} \gamma_{\l+1}^2 > \beta_{\l} \gamma_{\l-1}^2, \qquad \l \geq m+1,
\end{equation}
and if $(\delta_{\l})_{\l=0}^{\infty}$ is any sequence where $\delta_{\l} \in (0,1]$, then one has 
\begin{align*}
 &\sum_{\l=m}^{m+1} \alpha_{\l-1} \gamma_{\l}^2 \norm{X_- u_{\l}}^2 + \sum_{\l=m+2}^{\infty} (1-\delta_{\l-1})(\alpha_{\l-1} \gamma_{\l}^2 - \beta_{\l-1} \gamma_{\l-2}^2) \norm{X_- u_{\l}}^2+ \kappa \sum_{\l=m}^{\infty} \lambda_{\l} \gamma_{\l}^2 \norm{u_{\l}}^2 \\
 &\qquad + \sum_{\l=m}^{\infty} \gamma_{\l}^2 \norm{Z(u_{\l})}^2 \leq \sum_{\l=m+1}^{\infty} \left[ 1 + \frac{1-\delta_{\l}}{\delta_{\l}} \frac{\beta_{\l} \gamma_{\l-1}^2}{\alpha_{\l} \gamma_{\l+1}^2} \right] \frac{\alpha_{\l} \gamma_{\l+1}^2 \beta_{\l} \gamma_{\l-1}^2}{\alpha_{\l} \gamma_{\l+1}^2 - \beta_{\l} \gamma_{\l-1}^2} \norm{(Xu)_{\l}}^2. \qedhere
\end{align*}
whenever $u \in C^{\infty}_F(SM, \mC^n)$ (with $u|_{\partial(SM)} = 0$ in the boundary case).
\end{Proposition}

\begin{Remark}
The sequence $(\gamma_{\l})$ corresponds to the weights in the Carleman estimate, and the parameters $(\delta_{\l})$ fine-tune the weights for the $X_-$ terms on the left. Later we will essentially choose $\gamma_{\l} = \l^s$ for $s > -1/2$ and $\delta_{\l} \equiv 1/2$. An even simpler choice, which would also be sufficient for most of our purposes, would be to take $\delta_{\l} \equiv 1$. Then the above estimate becomes (after dropping the $X_- u_{\l}$ and $Z(u_{\l})$ terms) 
\[
\kappa \sum_{\l=m}^{\infty} \lambda_{\l} \gamma_{\l}^2 \norm{u_{\l}}^2 
 \leq \sum_{\l=m+1}^{\infty} \frac{\alpha_{\l} \gamma_{\l+1}^2 \beta_{\l} \gamma_{\l-1}^2}{\alpha_{\l} \gamma_{\l+1}^2 - \beta_{\l} \gamma_{\l-1}^2} \norm{(Xu)_{\l}}^2
\]
for any $u \in C^{\infty}_F(SM)$ (with $u|_{\partial(SM)} = 0$ in the boundary case)
\end{Remark}

\begin{proof}[Proof of Proposition \ref{prop_carleman_general_weight}]
Let $u \in C^{\infty}_F(SM, \mC^n)$, and let $(\gamma_{\l})_{\l=0}^{\infty}$ be a  sequence of positive real numbers satisfying \eqref{gammal_condition}. Lemma \ref{lemma_pestov_localized_negativecurvature} implies that for each $\l \geq 0$, 
\[
\alpha_{\l-1} \gamma_{\l}^2 \norm{X_- u_{\l}}^2 + \kappa \lambda_{\l} \gamma_{\l}^2 \norm{u_{\l}}^2 + \gamma_{\l}^2 \norm{Z(u_{\l})}^2 \leq \beta_{\l+1} \gamma_{\l}^2 \norm{X_+ u_{\l}}^2.
\]
Let now $\eps_{\l} > 0$ be positive numbers, and observe that 
\begin{align}
\norm{X_+ u_{\l}}^2 &= \norm{(Xu)_{\l+1} - X_- u_{\l+2}}^2 = \norm{(Xu)_{\l+1}}^2 - 2 \,\mathrm{Re} ((Xu)_{\l+1}, X_- u_{\l+2}) + \norm{X_- u_{\l+2}}^2 \notag \\
 &\leq \left(1+\frac{1}{\eps_{\l}}\right) \norm{(Xu)_{\l+1}}^2 + (1+\eps_{\l}) \norm{X_- u_{\l+2}}^2. \label{xplusul_epsl_estimate}
\end{align}
Combining the last two inequalities gives the estimate 
\begin{multline*}
\alpha_{\l-1} \gamma_{\l}^2 \norm{X_- u_{\l}}^2 + \kappa \lambda_{\l} \gamma_{\l}^2 \norm{u_{\l}}^2 + \gamma_{\l}^2 \norm{Z(u_{\l})}^2 \\
 \leq \beta_{\l+1} \gamma_{\l}^2 \left(1+\frac{1}{\eps_{\l}}\right) \norm{(Xu)_{\l+1}}^2 + \beta_{\l+1} \gamma_{\l}^2 (1+\eps_{\l}) \norm{X_- u_{\l+2}}^2.
\end{multline*}
Adding up these estimates for $m \leq \l \leq N$, where $m \geq 1$, yields 
\begin{multline}
\sum_{\l=m}^{N} \alpha_{\l-1} \gamma_{\l}^2 \norm{X_- u_{\l}}^2 + \sum_{\l=m}^{N} \kappa \lambda_{\l} \gamma_{\l}^2 \norm{u_{\l}}^2 + \sum_{\l=m}^{N} \gamma_{\l}^2 \norm{Z(u_{\l})}^2 \\
 \leq \sum_{\l=m+1}^{N+1} \beta_{\l} \gamma_{\l-1}^2 \left(1+\frac{1}{\eps_{\l-1}}\right) \norm{(Xu)_{\l}}^2 + \sum_{\l=m+2}^{N+2} \beta_{\l-1} \gamma_{\l-2}^2 (1+\eps_{\l-2}) \norm{X_- u_{\l}}^2. \label{pestov_localized_summed_negativecurvature}
\end{multline}
We would like to choose $(\gamma_{\l})$ and $(\eps_{\l})$ so that a large part of the last term on the right can be absorbed in the first term on the left. The minimal requirement is that 
\[
\beta_{\l-1} \gamma_{\l-2}^2 (1+\eps_{\l-2}) \leq \alpha_{\l-1} \gamma_{\l}^2, \qquad m+2 \leq \l \leq N.
\]
We will choose $(\eps_{\l})$ so that 
\begin{equation} \label{epsl_condition}
\eps_{\l-2} = \delta_{\l-1} \left[ \frac{\alpha_{\l-1}}{\beta_{\l-1}} \frac{\gamma_{\l}^2}{\gamma_{\l-2}^2} - 1 \right], \qquad \l \geq m+2,
\end{equation}
where $\delta_{\l-1} \in (0,1]$. In order for $\eps_{\l}$ to be positive, we need to have $\alpha_{\l-1} \gamma_{\l}^2 > \beta_{\l-1} \gamma_{\l-2}^2$ for $\l \geq m+2$, which follows from the assumption \eqref{gammal_condition}.

Using \eqref{epsl_condition}, we may express the weights on the right hand side of \eqref{pestov_localized_summed_negativecurvature} as 
\[
\beta_{\l-1} \gamma_{\l-2}^2 (1+\eps_{\l-2}) = \beta_{\l-1} \gamma_{\l-2}^2 + \delta_{\l-1}(\alpha_{\l-1} \gamma_{\l}^2 - \beta_{\l-1} \gamma_{\l-2}^2)
\]
and 
\begin{align*}
\beta_{\l} \gamma_{\l-1}^2 \left(1+\frac{1}{\eps_{\l-1}}\right) &= \beta_{\l} \gamma_{\l-1}^2 \left[ 1+\frac{1}{\delta_{\l}} \frac{\beta_{\l} \gamma_{\l-1}^2}{\alpha_{\l} \gamma_{\l+1}^2 - \beta_{\l} \gamma_{\l-1}^2} \right] \\
 &= \frac{\alpha_{\l} \gamma_{\l+1}^2 \beta_{\l} \gamma_{\l-1}^2}{\alpha_{\l} \gamma_{\l+1}^2 - \beta_{\l} \gamma_{\l-1}^2} \left[ 1 + \frac{1-\delta_{\l}}{\delta_{\l}} \frac{\beta_{\l} \gamma_{\l-1}^2}{\alpha_{\l} \gamma_{\l+1}^2} \right].
\end{align*}
Inserting these expressions in \eqref{pestov_localized_summed_negativecurvature}, it follows that 
\begin{align*}
 &\sum_{\l=m}^{m+1} \alpha_{\l-1} \gamma_{\l}^2 \norm{X_- u_{\l}}^2 + \sum_{\l=m+2}^{N} (1-\delta_{\l-1})(\alpha_{\l-1} \gamma_{\l}^2 - \beta_{\l-1} \gamma_{\l-2}^2) \norm{X_- u_{\l}}^2+ \kappa \sum_{\l=m}^{N} \lambda_{\l} \gamma_{\l}^2 \norm{u_{\l}}^2 \\
 &\qquad + \sum_{\l=m}^{N} \gamma_{\l}^2 \norm{Z(u_{\l})}^2 \leq \sum_{\l=m+1}^{N+1} \left[ 1 + \frac{1-\delta_{\l}}{\delta_{\l}} \frac{\beta_{\l} \gamma_{\l-1}^2}{\alpha_{\l} \gamma_{\l+1}^2} \right] \frac{\alpha_{\l} \gamma_{\l+1}^2 \beta_{\l} \gamma_{\l-1}^2}{\alpha_{\l} \gamma_{\l+1}^2 - \beta_{\l} \gamma_{\l-1}^2} \norm{(Xu)_{\l}}^2 \\
 &\qquad \qquad + \sum_{\l=N+1}^{N+2} \left[ \beta_{\l-1} \gamma_{\l-2}^2 + \delta_{\l-1}(\alpha_{\l-1} \gamma_{\l}^2 - \beta_{\l-1} \gamma_{\l-2}^2)\right] \norm{X_- u_{\l}}^2.
\end{align*}
Since $u$ has finite degree we can take the limit as $N \to \infty$, which proves the result.
\end{proof}

We will next prove Theorem \ref{thm_pestov_carleman_negativecurvature} 
by choosing suitable weights in Proposition \ref{prop_carleman_general_weight}. The proof will involve two elementary lemmas.

\begin{Lemma} \label{lemma_elementary_lower_bound}
If $s \geq 0$, then 
\[
(\l+1)^s - (\l-1)^s \geq s \l^{s-1}, \qquad \l \geq 1.
\]
\end{Lemma}
\begin{proof}
Writing $x = 1/\l$ and using that $\log(1-x) \leq -x$ and $\log(1+x) \geq x \log(2)$ for $x \in (0,1)$, we obtain 
\[
(1+x)^s-(1-x)^s \geq 2 \sinh(sx \log(2)) \geq 2 \log(2) sx.
\]
Since $2 \log(2) \geq 1$, the result follows.
\end{proof}

\begin{Remark}
By using the generalized binomial theorem, one can show that the sharp constant on the right hand side of Lemma \ref{lemma_elementary_lower_bound} is $\min(2s, 2^s)$.
\end{Remark}

\begin{Lemma} \label{lemma_factorial_estimate}
For any $\l \geq 2$ one has 
\[
\frac{1}{4 \sqrt{\l}} \leq \prod_{j=0}^{[\frac{\l-2}{2}]} \frac{\l-1-2j}{\l-2j} \leq \frac{4}{\sqrt{\l}}.
\]
\end{Lemma}
\begin{proof}
We do the proof for $\l = 2N$ even (the odd case is similar). Note that 
\[
\frac{2N-1}{2N} \  \frac{2N-3}{2N-2} \, \cdots \, \frac{1}{2} = \frac{(2N)!}{2^{2N} (N!)^2}.
\]
The Stirling approximation $\sqrt{2\pi N} \left( \frac{N}{e} \right)^N \leq N! \leq \sqrt{2\pi N} \left( \frac{N}{e} \right)^N e^{\frac{1}{12 N}}$ yields that 
\[
e^{-\frac{1}{6 N}} \frac{\sqrt{2\pi \cdot 2N}}{(\sqrt{2\pi N})^2} \leq \frac{(2N)!}{2^{2N} (N!)^2} \leq \frac{\sqrt{2\pi \cdot 2N}}{(\sqrt{2\pi N})^2} e^{\frac{1}{24 N}}.
\]
This proves the result.
\end{proof}

\begin{proof}[Proof of Theorem \ref{thm_pestov_carleman_negativecurvature}]
Define a sequence $(\ol{\gamma}_\l)$ by 
\[
\ol{\gamma}_1 = \ol{\gamma}_2 = 1, \qquad \ol{\gamma}_{\l+1}^2 = \frac{\beta_\l}{\alpha_\l} \ol{\gamma}_{\l-1}^2 \text{ for $\l \geq 2$}.
\]
Choose $\gamma_{\l} = \ol{\gamma}_\l \l^{\sigma/2}$ in Proposition \ref{prop_carleman_general_weight}, where $\sigma > 0$. By Lemma \ref{lemma_elementary_lower_bound} we have  
\begin{equation} \label{weight_est1}
\alpha_{\l} \gamma_{\l+1}^2 - \beta_{\l} \gamma_{\l-1}^2 = \beta_{\l} \ol{\gamma}_{\l-1}^2 \left[ (\l+1)^{\sigma} - (\l-1)^{\sigma} \right] \geq \beta_{\l} \ol{\gamma}_{\l-1}^2 \left[ \sigma \l^{\sigma-1} \right]
\end{equation}
and 
\begin{equation} \label{weight_est2}
\frac{\alpha_{\l} \gamma_{\l+1}^2 \beta_{\l} \gamma_{\l-1}^2}{\alpha_{\l} \gamma_{\l+1}^2 - \beta_{\l} \gamma_{\l-1}^2} \leq \frac{(\beta_{\l} \ol{\gamma}_{\l-1}^2)^2 (\l^2-1)^{\sigma}}{\beta_{\l} \ol{\gamma}_{\l-1}^2 \left[ \sigma \l^{\sigma-1} \right]} \leq \beta_\l \ol{\gamma}_{\l-1}^2 \frac{\l^{\sigma+1}}{\sigma}.
\end{equation}

We will next estimate $\ol{\gamma}_{\l-1}^2$. By definition we have 
\[
\ol{\gamma}_{\l+1}^2 = \frac{\beta_{\l}}{\alpha_{\l}} \ol{\gamma}_{\l-1}^2 = \ldots = \prod_{j=0}^{[\frac{\l-2}{2}]} \frac{\beta_{\l-2j}}{\alpha_{\l-2j}}.
\]
Using the definition of $\alpha_\l$ and $\beta_\l$ we further have 
\[
\frac{\beta_{\l}}{\alpha_{\l}} = \frac{(\l-1)(\l+d-2)}{\l(\l+d-1)}.
\]
Writing 
\[
g_\l = \prod_{j=0}^{[\frac{\l-2}{2}]} \frac{\l-1-2j}{\l-2j},
\]
it follows that 
\[
\ol{\gamma}_{\l+1}^2 = \frac{g_\l g_{\l+d-2}}{g_{\tilde{d}}}
\]
where $\tilde{d} = d-1$ if $d$ is even, and $\tilde{d} = d$ if $d$ is odd. Using Lemma \ref{lemma_factorial_estimate} we obtain 
\begin{equation} \label{weight_est3}
\frac{c}{\l} \leq \ol{\gamma}_{\l}^2 \leq \frac{C}{\l}, \qquad \l \geq 1,
\end{equation}
where $c$ and $C$ only depend on $d$.

We now choose $\delta_\l = 1/2$ in Proposition \ref{prop_carleman_general_weight} and observe, using \eqref{weight_est1}--\eqref{weight_est3}, that for $\l \geq 1$ one has 
\begin{align*}
\alpha_{\l-1} \gamma_{\l}^2 &\geq c \l^{\sigma}, \\
\alpha_{\l-1} \gamma_{\l}^2 - \beta_{\l-1} \gamma_{\l-2}^2 &\geq c \sigma (\l-1)^{\sigma-1}, \\
\lambda_\l \gamma_{\l}^2 &\geq c \l^{\sigma+1}, \\
\gamma_{\l}^2 &\geq c \l^{\sigma-1},
\end{align*}
and 
\begin{align*}
\frac{\beta_\l \gamma_{\l-1}^2}{\alpha_\l \gamma_{\l+1}^2} &\leq 1, \\
\frac{\alpha_{\l} \gamma_{\l+1}^2 \beta_{\l} \gamma_{\l-1}^2}{\alpha_{\l} \gamma_{\l+1}^2 - \beta_{\l} \gamma_{\l-1}^2} &\leq C \frac{\l^{\sigma+1}}{\sigma}
\end{align*}
for $C, c > 0$ depending on $d$. The theorem follows upon choosing $\sigma = 2s+1$.
\end{proof}

\section{Carleman estimates in nonpositive curvature} \label{sec_nonpositive_curvature}

In this section we discuss Carleman estimates in the case of nonpositive sectional curvature instead of strictly negative sectional curvature. The extra positive term in Lemma \ref{lemma_pestov_localized_negativecurvature} becomes zero in the case, but one gains some positivity from the fact that there are no conjugate points (in a suitable strong sense). Thus, in this section $(M,g)$ will be a compact Riemannian manifold with or without boundary, and we will assume that $(M,g)$ is simple/Anosov with nonpositive sectional curvature.

The main result is the following Carleman estimate with linear instead of logarithmic weights.

\begin{Theorem} \label{thm_pestov_carleman_nonpositivecurvature}
Let $(M,g)$ be a simple/Anosov manifold having nonpositive sectional curvature. There exist $m_0, \tau_0, \kappa > 0$ such that for any $\tau \geq \tau_0$, for any $m \geq m_0$ and for any $u \in C^{\infty}_F(SM, \mC^n)$ (with $u|_{\partial(SM)} = 0$ in the boundary case) one has 
\[
\sum_{\l=m}^{\infty} e^{2\tau \phi_\l} \norm{u_\l}^2 \leq \frac{24}{\kappa e^{2 \tau}} \sum_{\l=m+1}^{\infty} e^{2\tau \phi_\l} \norm{(Xu)_\l}^2
\]
where $\phi_\l = \l$.
\end{Theorem}

The constant $\kappa > 0$ in the above theorem comes from the fact that compact/simple manifolds with nonpositive curvature have no conjugate points. It is the same constant that appears in the following lemma
proved in \cite[Theorem 7.2 and Lemma 11.2]{PSU_hd}:

\begin{Lemma} \label{lemma_simple_anosov_controlled} 
If $(M,g)$ is a compact simple/Anosov manifold, there exist $\alpha \in (0,1]$ and $\kappa > 0$ so that 
\[
\norm{XZ}^2 - (R Z, Z) \geq \alpha\norm{XZ}^2 +\kappa \norm{Z}^2
\]
for any $Z \in \mathcal{Z}^n$ (with $Z|_{\partial(SM)} = 0$ in the boundary case).
\end{Lemma}


We begin the proof of Theorem \ref{thm_pestov_carleman_nonpositivecurvature} with the following result, which is a counterpart of Lemma \ref{lemma_pestov_localized_negativecurvature}.

\begin{Lemma} \label{lemma_pestov_localized_nonpositivecurvature}
Let $(M,g)$ be a simple/Anosov manifold having nonpositive sectional curvature. There is $\kappa > 0$ such that for any $\l \geq 1$ and for any $\sigma_\l$ with $0 \leq \sigma_\l \leq 1$ one has 
\begin{multline*}
\left( \alpha_{\l-1} - \sigma_\l \lambda_{\l-1}\left(1+\frac{1}{\l+d-3}\right)^2 \right) \norm{X_- u}^2 + \sigma_\l \kappa \lambda_\l \norm{u}^2 \\
 \leq \left(\beta_{\l+1} + \sigma_\l \lambda_{\l+1}\left(1-\frac{1}{\l+1}\right)^2\right) \norm{X_+ u}^2, \qquad u \in \Omega_\l,
\end{multline*}
if additionally $u|_{\partial(SM)} = 0$ in the boundary case.
\end{Lemma}
\begin{proof}
We use the standard Pestov identity (Lemma \ref{lemma_pestov_general_connection}) with $A = 0$: 
\[
\norm{X \vd u}^2 - (R \vd u, \vd u) + (d-1) \norm{Xu}^2 = \norm{\vd X u}^2.
\]
The nonpositive curvature assumption implies that 
\[
\norm{X \vd u}^2 - (R \vd u, \vd u) \geq \norm{X \vd u}^2.
\]
On the other hand, the simple/Anosov assumption together with Lemma \ref{lemma_simple_anosov_controlled} implies that there is $\kappa > 0$ (we omit the $\alpha \norm{X \vd u}^2$ term) such that 
\[
\norm{X \vd u}^2 - (R \vd u, \vd u) \geq \kappa \norm{\vd u}^2.
\]
Interpolating the last two inequalities yields that for any $\sigma_\l$ with $0 \leq \sigma_\l \leq 1$, 
\[
\norm{X \vd u}^2 - (R \vd u, \vd u) \geq (1-\sigma_\l) \norm{X \vd u}^2 + \sigma_\l \kappa \norm{\vd u}^2.
\]
Inserting this estimate in the Pestov identity implies that 
\begin{equation} \label{pestov_sigmal_localized}
(1-\sigma_\l) \norm{X \vd u}^2 + \sigma_\l \kappa \norm{\vd u}^2 \leq \norm{\vd X u}^2 - (d-1) \norm{Xu}^2.
\end{equation}

We now assume that $u = u_\l \in \Omega_\l$ and simplify \eqref{pestov_sigmal_localized} as in the proof of Lemma \ref{lemma_pestov_xplusminus_generalconnection}. This gives 
\[
\alpha_{\l-1} \norm{X_- u_\l}^2 + \norm{Z(u_\l)}^2 + \sigma_\l \kappa \norm{\vd u}^2 \leq \beta_{\l+1} \norm{X_+ u_\l}^2 + \sigma_\l \norm{X \vd u}^2
\]
Again as in the proof of Lemma \ref{lemma_pestov_xplusminus_generalconnection}, one has 
\begin{align*}
\norm{X \vd u_\l}^2 &= \norm{\vd \left[ \left(1-\frac{1}{\l+1}\right) X_+ u_\l + \left(1+\frac{1}{\l+d-3}\right) X_- u_\l \right] - Z(u_\l)}^2 \\
 &= \lambda_{\l+1}\left(1-\frac{1}{\l+1}\right)^2 \norm{X_+ u_\l}^2 + \lambda_{\l-1}\left(1+\frac{1}{\l+d-3}\right)^2 \norm{X_- u_\l}^2 + \norm{Z(u_\l)}^2.
\end{align*}
The result follows by combining the last two formulas and using the trivial estimate $(1-\sigma_\l) \norm{Z(u_\l)}^2 \geq 0$.
\end{proof}

The estimate in Lemma \ref{lemma_pestov_localized_nonpositivecurvature} is of course only useful when the coefficient in front of $\norm{X_- u}^2$ is nonnegative. For large $\l$ one has 
\[
\alpha_{\l-1} - \sigma_\l \lambda_{\l-1} \left(1+\frac{1}{\l+d-3} \right)^2 \sim 2\l - \sigma_\l \l^2,
\]
so one needs to have $\sigma_\l \leq \frac{2}{\l}$ for $\l$ large. (If one keeps the $\alpha \norm{X \vd u}^2$ that was omitted in the proof of Lemma \ref{lemma_pestov_localized_nonpositivecurvature}, the previous condition becomes $\sigma_\l \leq \frac{2}{(1-\alpha)\l}$ so one still needs $\sigma_\l \lesssim \l^{-1}$ for $\l$ large.)

We next give the counterpart of Proposition \ref{prop_carleman_general_weight} (for $\delta_\l \equiv 1$), which corresponds to a Carleman estimate with general weights.

\begin{Proposition} \label{prop_carleman_general_weight_nonpositive_curvature}
Assume that $(M,g)$ is a simple/Anosov manifold of nonpositive sectional curvature. If $(\gamma_\l)$ and $(\sigma_\l)$ are sequences satisfying $\gamma_\l > 0$, $0 < \sigma_\l < 1$, and 
\begin{equation} \label{gammal_condition_nonpositive_curvature}
\left(\alpha_{\l} - \sigma_{\l+1} \lambda_{\l}\left(1+\frac{1}{\l+d-2}\right)^2 \right) \gamma_{\l+1}^2 > \left(\beta_{\l} + \sigma_{\l-1} \lambda_{\l}\left(1-\frac{1}{\l}\right)^2 \right) \gamma_{\l-1}^2
\end{equation}
then one has 
\begin{multline*}
\kappa \sum_{\l=m}^{\infty} \sigma_\l \lambda_\l \gamma_\l^2 \norm{u_\l}^2 \\
 \leq \sum_{\l=m+1}^{\infty} \frac{(\alpha_{\l} - \sigma_{\l+1} \lambda_{\l}(1+\frac{1}{\l+d-2})^2) \gamma_{\l+1}^2 (\beta_{\l} + \sigma_{\l-1} \lambda_{\l}(1-\frac{1}{\l})^2) \gamma_{\l-1}^2}{(\alpha_{\l} - \sigma_{\l+1} \lambda_{\l}(1+\frac{1}{\l+d-2})^2) \gamma_{\l+1}^2 - (\beta_{\l} + \sigma_{\l-1} \lambda_{\l}(1-\frac{1}{\l})^2) \gamma_{\l-1}^2} \norm{(Xu)_{\l}}^2
\end{multline*}
whenever $m \geq 1$ and $u \in C^{\infty}_F(SM, \mC^n)$ (with $u|_{\partial(SM)} = 0$ in the boundary case).
\end{Proposition}
\begin{proof}
Multiplying the estimate in Lemma \ref{lemma_pestov_localized_nonpositivecurvature} by $\gamma_\l^2$ and using \eqref{xplusul_epsl_estimate}, we obtain 
\begin{multline*}
\left(\alpha_{\l-1} - \sigma_\l \lambda_{\l-1}\left(1+\frac{1}{\l+d-3}\right)^2\right) \gamma_\l^2 \norm{X_- u_\l}^2 + \sigma_\l \kappa \lambda_\l \gamma_\l^2 \norm{u_\l}^2 \\
 \leq \left(\beta_{\l+1} + \sigma_\l \lambda_{\l+1}\left(1-\frac{1}{\l+1}\right)^2\right) \gamma_\l^2 \left[ \left(1+\frac{1}{\eps_\l}\right) \norm{(Xu)_{\l+1}}^2 + (1+\eps_\l) \norm{X_- u_{\l+2}}^2 \right].
\end{multline*}
Summing up these estimates for $\l \geq m$, where $m \geq 1$, yields 
\begin{multline*}
\sum_{\l=m}^{\infty} \left(\alpha_{\l-1} - \sigma_\l \lambda_{\l-1}\left(1+\frac{1}{\l+d-3} \right)^2 \right) \gamma_\l^2 \norm{X_- u_\l}^2 + \sum_{\l=m}^{\infty} \sigma_\l \kappa \lambda_\l \gamma_\l^2 \norm{u_\l}^2 \\
 \leq \sum_{\l=m+1}^{\infty} \left(\beta_{\l} + \sigma_{\l-1} \lambda_{\l}\left(1-\frac{1}{\l}\right)^2 \right) \gamma_{\l-1}^2 \left( 1+\frac{1}{\eps_{\l-1}} \right) \norm{(Xu)_{\l}}^2 \\
 + \sum_{\l=m+2}^{\infty} \left(\beta_{\l-1} + \sigma_{\l-2} \lambda_{\l-1}\left(1-\frac{1}{\l-1}\right)^2 \right) \gamma_{\l-2}^2 \left(1+\eps_{\l-2}\right) \norm{X_- u_{\l}}^2.
\end{multline*}
We now choose $\eps_\l$ so that for $\l \geq m+2$, 
\[
1 + \eps_{\l-2} = \frac{(\alpha_{\l-1} - \sigma_\l \lambda_{\l-1}(1+\frac{1}{\l+d-3})^2) \gamma_\l^2}{(\beta_{\l-1} + \sigma_{\l-2} \lambda_{\l-1}(1-\frac{1}{\l-1})^2) \gamma_{\l-2}^2}.
\]
It follows that 
\[
\frac{1}{\eps_{\l-1}} = \frac{(\beta_{\l} + \sigma_{\l-1} \lambda_{\l}(1-\frac{1}{\l})^2) \gamma_{\l-1}^2}{(\alpha_{\l} - \sigma_{\l+1} \lambda_{\l}(1+\frac{1}{\l+d-2})^2) \gamma_{\l+1}^2 - (\beta_{\l} + \sigma_{\l-1} \lambda_{\l}(1-\frac{1}{\l})^2) \gamma_{\l-1}^2}.
\]
The result follows.
\end{proof}

We will now prove the Carleman estimate by using suitable choices for the weights $\sigma_\l$ and $\gamma_\l$.

\begin{proof}[Proof of Theorem \ref{thm_pestov_carleman_nonpositivecurvature}]
Fix a constant $\mu > 1$, choose $\gamma_\l$ so that $\l \gamma_\l^2 = \mu^\l$, and choose $\sigma_\l = \frac{2\delta}{\l}$ where $0 < \delta < 1$ is fixed. The estimate in Proposition \ref{prop_carleman_general_weight_nonpositive_curvature} implies that  
\begin{multline*}
2 \kappa \delta  \sum_{\l=m}^{\infty} \mu^\l \norm{u_\l}^2 = \kappa \sum_{\l=m}^{\infty} \sigma_\l \l  \mu^\l \norm{u_\l}^2 \leq \kappa \sum_{\l=m}^{\infty} \sigma_\l \frac{\lambda_\l}{\l} \mu^\l \norm{u_\l}^2 = \kappa \sum_{\l=m}^{\infty} \sigma_\l \lambda_\l \gamma_\l^2 \norm{u_\l}^2 \\
 \leq \sum_{\l=m+1}^{\infty} \frac{(\alpha_{\l} - \sigma_{\l+1} \lambda_{\l}(1+\frac{1}{\l+d-2})^2) \gamma_{\l+1}^2 (\beta_{\l} + \sigma_{\l-1} \lambda_{\l}(1-\frac{1}{\l})^2) \gamma_{\l-1}^2}{(\alpha_{\l} - \sigma_{\l+1} \lambda_{\l}(1+\frac{1}{\l+d-2})^2) \gamma_{\l+1}^2 - (\beta_{\l} + \sigma_{\l-1} \lambda_{\l}(1-\frac{1}{\l})^2) \gamma_{\l-1}^2} \norm{(Xu)_{\l}}^2.
\end{multline*}
Using the formulas in Lemma \ref{lemma_pestov_xplusminus_generalconnection} and the fact that $\sigma_\l = \frac{2\delta}{\l}$, we have 
\begin{align*}
 &\alpha_{\l} - \sigma_{\l+1} \lambda_{\l}\left(1+\frac{1}{\l+d-2}\right)^2 \\
 &= (2\l+d-2) \left(1+\frac{1}{\l+d-2} \right) - \frac{2\delta}{\l+1} \lambda_\l \left(1+\frac{1}{\l+d-2}\right)^2 \\
 &= 2(1-\delta)\l + O(1)
\end{align*}
and 
\begin{align*}
 &\beta_{\l} + \sigma_{\l-1} \lambda_{\l}\left(1-\frac{1}{\l}\right)^2 \\
 &=  (2\l+d-2) \left(1-\frac{1}{\l} \right) + \frac{2 \delta}{\l-1} \lambda_{\l} \left(1-\frac{1}{\l}\right)^2 \\
 &= 2(1+\delta)\l + O(1)
\end{align*}
when $\l$ is large. These estimates also imply that 
\begin{align*}
\left(\alpha_{\l} - \sigma_{\l+1} \lambda_{\l}\left(1+\frac{1}{\l+d-2}\right)^2\right) \gamma_{\l+1}^2 &= (2(1-\delta) + O(\l^{-1})) \mu^{\l+1}, \\
\left(\beta_{\l} + \sigma_{\l-1} \lambda_{\l}\left(1-\frac{1}{\l}\right)^2\right) \gamma_{\l-1}^2 &= (2(1+\delta) + O(\l^{-1})) \mu^{\l-1}
\end{align*}
for $\l$ large.

Using the above bounds and choosing $m$ large enough, we obtain the estimate 
\begin{align*}
 &2 \kappa \delta \sum_{\l=m}^{\infty} \mu^\l \norm{u_\l}^2 \\
 &\leq \sum_{\l=m+1}^{\infty} \frac{(2(1-\delta) + O(\l^{-1})) (2(1+\delta) + O(\l^{-1})) }{(2(1-\delta) + O(\l^{-1})) - (2(1+\delta) + O(\l^{-1})) \mu^{-2}}  \mu^{\l-1} \norm{(Xu)_{\l}}^2 \\
 &\leq \frac{8 (1-\delta) (1+\delta)}{(1-\delta)  - 4(1+\delta) \mu^{-2}} \sum_{\l=m+1}^{\infty} \mu^{\l-1} \norm{(Xu)_{\l}}^2.
\end{align*}
This estimate makes sense if $1-\delta > 4(1+\delta) \mu^{-2}$, i.e.\ $\mu^2 > \frac{4(1+\delta)}{1-\delta}$.

Writing $\mu^2 = \frac{4(1+\delta)}{1-\delta} a$ where $a > 1$, the above estimate becomes 
\begin{align*}
2 \kappa \delta \sum_{\l=m}^{\infty} \mu^\l \norm{u_\l}^2 \leq \frac{8 (1+\delta)}{1-1/a} \sum_{\l=m+1}^{\infty} \mu^{\l-1} \norm{(Xu)_{\l}}^2.
\end{align*}
We will now make the choices 
\[
\delta = \frac{1}{2}, \qquad a = 2, \qquad \mu = e^{2\tau}
\]
where $\tau \geq \tau_0$ with $\tau_0$ chosen so that $e^{4\tau_0} > 12$. Then one has $\mu^2 > \frac{4(1+\delta)}{1-\delta}$, and the estimate takes the required form  
\[
\kappa \sum_{\l=m}^{\infty} e^{2\tau \l} \norm{u_\l}^2 \leq \frac{24}{e^{2\tau}} \sum_{\l=m+1}^{\infty} e^{2\tau \l} \norm{(Xu)_{\l}}^2. \qedhere
\]
\end{proof}

\begin{Remark}
So far, we have considered the case where $(M,g)$ has nonpositive curvature. One could ask if it is possible to obtain weighted estimates whenever $(M,g)$ is a general compact simple/Anosov manifold. In this case the sectional curvatures may be positive, and the most efficient way seems to be to write the Pestov identity as 
\[
\norm{X\vd u}^2 - (R \vd u, \vd u) + (d-1) \norm{Xu}^2 = \norm{\vd Xu}^2
\]
and use the fact that, by Lemma \ref{lemma_simple_anosov_controlled}, 
\[
\norm{X\vd u}^2 - (R \vd u, \vd u) \geq \alpha \norm{X\vd u}^2 + \kappa \norm{\vd u}^2
\]
for some $\alpha, \kappa > 0$. If one does this and runs the argument as in the negative curvature case, one arrives at the following estimate for any $u \in C^{\infty}(SM)$ (with $u|_{\partial(SM)} = 0$ in the boundary case):
\begin{multline*}
\alpha \sum_{\l=0}^{\infty} \alpha_\l \norm{X_- u_{\l+1}}^2 + \kappa \sum_{\l=1}^{\infty} \lambda_\l \norm{u_\l}^2 \\
 \leq (1-\alpha) \sum_{\l=2}^{\infty} (\lambda_\l - (d-1)) \norm{(Xu)_\l}^2 + \alpha \sum_{\l=2}^{\infty} \beta_\l \norm{X_+ u_{\l-1}}^2.
\end{multline*}

The $X_+$ terms on the right can be absorbed in the $X_-$ terms on the left, even with weights, just as in the negative curvature case. However, the new $(Xu)_\l$ terms on the right present problems. If one considers $u$ with $u_\l = 0$ for $\l < m$, which is the case relevant for the X-ray transform on $m$-tensors, then the $(Xu)_{m-1} = X_- u_m$ and $(Xu)_m = X_- u_{m+1}$ terms on the right cannot be absorbed to the left if $m \geq 2$ is large. Moreover, adding weights (i.e.\ replacing $u_\l$ by $\gamma_\l u_\l$) does not help, since the $X_- u_m$ term on the right is multiplied by the same $\gamma_m^2$ as the $X_- u_m$ and $u_m$ terms on the left. Thus this method does not seem to shed new light to the tensor tomography problem of simple/Anosov manifolds (which is still open when $\dim(M) \geq 3$ and $m \geq 2$).

One could still ask if one gets new weighted estimates in the cases $m=0,1$. In these cases the $(Xu)_0$ and $(Xu)_1$ terms pose no problem (they are not present on the right). However, replacing $u$ by its weighted version $\sum \gamma_\l u_\l$ in the $(Xu)_\l$ terms on the right will generate new $X_-$ and $X_+$ terms, and it seems that absorbing the new terms from right to left requires weights that grow at most mildly. This is not sufficient for dealing with large general connections or Higgs fields, but could lead to a shifted version of the Pestov identity as in Theorem \ref{thm_pestov_carleman_shifted} but where $s$ has to be in the range $-1/2 < s \leq 1/2$. We omit the details.
\end{Remark}

\section{A regularity result for the transport equation} \label{sec_regularity}

Let $(M,g)$ be a compact Riemannian manifold with boundary whose geodesic flow is {\it nontrapping}, and
let $\mathcal A:SM\to \C^{n\times n}$ be an arbitrary smooth attenuation. If the boundary of $M$ is strictly convex, it was proved in \cite[Proposition 5.2]{PSU2} that any solution to $Xu + \mathcal{A}u = -f$ in $SM$ with $u|_{\partial(SM)} = 0$ is smooth in $SM$ whenever $f$ is smooth. We would like to prove an analogue
of this statement, but without assuming that the boundary of $M$ is strictly convex.
We will do so at the cost of assuming that the function $f$ is supported in the interior of $SM$.

Given a smooth
$f:SM\to\C^n$, denote by $u^f$ the unique solution to
\[\left\{\begin{array}{ll}
Xu+\mathcal{A}u=-f,\\
\,\,\,u|_{\partial_{-}(SM)}=0.\\
\end{array}\right.\]
The function $u^f$ may fail to be differentiable due to the non-smoothness of $\tau$. In fact, $\tau$ might not even
be continuous.

\begin{Proposition} \label{prop:smooth}
Let $f:SM\to\C^n$ be smooth, supported in the interior of $SM$ with $I_{\mathcal A}(f)=u^{f}|_{\partial_{+}(SM)}=0$. Then $u^f:SM\to \C^n$ is smooth.
\end{Proposition}

\begin{proof}  We consider $(M,g)$ isometrically embedded in a closed manifold $(N,g)$ and we extend $f$ by zero. We also extend $\mathcal{A}$ smoothly to $N$. 
The idea is to use the argument in \cite[Lemma 2.3]{Da06}
and replace $\tau$ locally by other suitable smooth functions. 
Let $\rho(x,v)$ be any time such that $\gamma_{x,v}(\rho(x,v))\notin M$. Consider the function $w(t,x,v)$ defined by the ordinary differential equation

\[\left\{\begin{array}{ll}
\frac{d}{dt }w(t,x,v)+\mathcal{A}(\varphi_{t}(x,v))w(t,x,v)=-f(\varphi_{t}(x,v)),\\
w(\rho(x,v),x,v)=0.\\
\end{array}\right.\]
We claim that for $(x,v)\in SM$
\begin{equation}
u^{f}(x,v)=w(0,x,v).
\label{eq:h}
\end{equation}
To prove \eqref{eq:h} observe that the function $u(t,x,v):=u^{f}(\varphi_{t}(x,v))$ solves the differential equation
\[\left\{\begin{array}{ll}
\frac{d}{dt }u(t,x,v)+\mathcal{A}(\varphi_{t}(x,v))u(t,x,v)=-f(\varphi_{t}(x,v)),\\
u(\tau(x,v),x,v)=0.\\
\end{array}\right.\]
Next note that the interval $[\tau,\rho]$ can be decomposed into two types of intervals: 
the first type consists of intervals where the geodesic $\gamma_{x,v}$ is a maximal geodesic segment in $M$ with boundary points on $\partial M$, and the second type consists of intervals where $\gamma_{x,v}$ runs outside $M$.
Recall that outside $M$, $f=0$ and while inside $M$, $I_{\mathcal A}(f)=0$. Thus $w(\tau(x,v),x,v)=0$.
Hence \eqref{eq:h} follows from uniqueness of solutions of ordinary differential equations.

Since $\gamma_{x,v}(\rho(x,v))\notin M$, we can take a small hypersurface $\Sigma$ in $N$ transversally intersecting $\gamma_{x,v}$ at the point $\gamma_{x,v}(\rho(x,v))$ and disjoint from $M$. Then there is a neighbourhood of $(x,v)$ in $SN$ such that for every $\theta$ in this neighbourhood the geodesic $\gamma_{\theta}$ will hit $\Sigma$ at the time $\rho(\theta)$ smoothly depending on $\theta$. 
If we now use this local function $\rho(\theta)$ to define $w$, using \eqref{eq:h} we see that $u^f$ is smooth around $(x,v)$ since $w(0,\theta)$ is.
\end{proof}

\begin{Remark}{\rm All that is needed is that $f$ has vanishing jet at the boundary. 
It might be possible to prove that if $f=f(x)$ and $I_{\mathcal A}(f)=0$, then $f$ has vanishing jet at the boundary following the ideas in \cite{GMT} or \cite{SU09}, but we do not pursue this matter here.
 }
\end{Remark}

\section{Applications} \label{sec_applications}

Finally, we will employ the Carleman estimates proved in this article in various applications as described in the introduction.

\vspace{15pt}

\noindent {\bf Uniqueness results for the transport equation and X-ray transforms.} We start with the proof
of Theorem \ref{thm_xray_transport_intro}.

\begin{proof}[Proof of Theorem \ref{thm_xray_transport_intro}]
Suppose that $f$ has degree $m_0 \geq 0$, and let $m \geq \max\{ \l_0-1, m_0 \}$. Then $(Xu)_\l = -(\mathcal{A}(u))_\l$ for $\l \geq m+1$ and 
\[
\norm{(Xu)_\l} \leq R ( \norm{u_{\l-1}} + \norm{u_\l} + \norm{u_{\l+1}} ), \qquad \l \geq m+1.
\]
Inserting this estimate in Theorem \ref{thm_carleman_negativecurvature_intro} yields that, for $\tau \geq 1$,  
\begin{align*}
\sum_{\l=m}^{\infty} \l^{2\tau} \norm{u_\l}^2 &\leq \frac{CR}{\tau} \sum_{\l=m+1}^{\infty} \l^{2\tau} ( \norm{u_{\l-1}}^2 + \norm{u_\l}^2 + \norm{u_{\l+1}}^2 ) \\
 &\leq \frac{CR}{\tau} \sum_{\l=m}^{\infty} (\l+1)^{2\tau} \norm{u_\l}^2
\end{align*}
where $C = C_{d,\kappa}$. If we additionally assume that $m \geq 2\tau$, then 
\[
(\l+1)^{2\tau} = \left( 1 + \frac{1}{\l} \right)^{2\tau} \l^{2\tau} \leq e \l^{2\tau}, \qquad \l \geq m.
\]
Thus if $m \geq \max\{ \l_0-1, m_0, 2\tau \}$ we have 
\[
\sum_{\l=m}^{\infty} \l^{2\tau} \norm{u_\l}^2 \leq \frac{CR}{\tau} \sum_{\l=m}^{\infty} \l^{2\tau} \norm{u_\l}^2.
\]
Now fix $\tau = \tau_0 = \max \{ 2CR, 1 \}$ and let $m = \max\{ \l_0-1, m_0, 2\tau_0 \}$. Then the right hand side of the last inequality can be absorbed to the left, and we obtain 
\[
\sum_{\l=m}^{\infty} \l^{2\tau} \norm{u_\l}^2 \leq 0.
\]
Thus $u_\l = 0$ for $\l \geq m$, so $u$ has degree $\leq m-1$ where $m = \max \{ \l_0-1, m_0, 4 C R \}$ where $C = C_{d,\kappa}$.
\end{proof}

We next show Theorem \ref{thm_intro_attenuated_xray}.

\begin{proof}[Proof of Theorem \ref{thm_intro_attenuated_xray}]
Let $u = u^f$ be the solution of 
\[
(X+A+\Phi)u = -f \text{ in $SM$}, \qquad u|_{\partial_-(SM)} = 0.
\]
Since the attenuated X-ray transform of $f$ vanishes, we have $u|_{\partial(SM)} = 0$. This implies that $u \in C^{\infty}(SM, \mC^n)$; 
if $M$ has strictly convex boundary this follows from \cite[Proposition 5.2]{PSU2}, and alternatively if $f$ is supported in the interior of $SM$ we may use Proposition \ref{prop:smooth}. We are now exactly in the setting of Theorem \ref{thm_xray_transport_intro} with $\mathcal{A}(u) = Au + \Phi u$ and $l_0 = \max\{ m-1, 2 \}$, and since $f$ has finite degree it follows that also $u$ has finite degree, i.e.\ $u_l = 0$ for $l \geq m_0+1$ for some $m_0$.

We still need to show that $m_0 \leq m$. If this is not the case, we use the transport equation $Xu + Au + \Phi u = -f$ again to derive
$X^{A}_{+}u_{m_{0}}=0$. Since $u_{m_{0}}|_{\partial(SM)}=0$, \cite[Theorem 5.2]{GPSU} implies that
$u_{m_{0}}=0$ and arguing in the same way with $u_{m_{0}-1}$ and so on, we deduce that all $u_{l}=0$ for $l\geq m$. This shows that $f = -(X+A+\Phi) u$ where $u$ has degree $m-1$ and $u|_{\partial(SM)} = 0$.
\end{proof}

\vspace{5pt}

\noindent {\bf Scattering rigidity for general connections in negative curvature.}
Let us move to our next application, concerning an inverse problem with scattering data. On a nontrapping compact manifold $(M,g)$ with strictly convex boundary, the scattering relation $\alpha = \alpha_g: \partial_+(SM) \to \partial_-(SM)$ maps a starting point and direction of a geodesic to the end point and direction. If $(M,g)$ is simple, then knowing $\alpha_g$ is equivalent to knowing the boundary distance function $d_g$ which encodes the distances between any pair of boundary points \cite{Mi}. On two dimensional simple manifolds, the boundary distance function $d_g$ determines the metric $g$ up to an isometry which fixes the boundary \cite{PestovUhlmann}.

Given a connection $A$ and a potential $\Phi$ on the bundle $M \times \C^n$, there is an additional piece of scattering data. Consider the unique matrix solution
$U:SM\to \GL(n,\C)$ to the transport equation
$$
X U + A U + \Phi U = 0 \text{ in } SM, \quad U|_{\partial_{+}(SM)}=\id.
$$
As discussed in the introduction, the scattering data (or relation) corresponding to a matrix attenuation pair $(A,\Phi)$ in $(M,g)$ is the map
$$
C_{A,\Phi}:\partial_{-}(SM)\to \GL(n,\C), \ \  C_{A,\Phi} := U|_{\partial_{-}(SM)}.
$$
Theorem \ref{thm_intro_scattering_rigidity} in the introduction proves that knowledge of $C_{A,\Phi}$ uniquely determines the pair $(A,\Phi)$ up to gauge.
The theorem can be proved by introducing a {\it pseudo-linearization} that reduces the nonlinear problem to the linear one in Theorem \ref{thm_intro_attenuated_xray}.  The argument is identical to the one used in \cite{PSU2} and \cite{PSUZ16}, and we do not need to repeat it here.

\bigskip

\noindent {\bf Transparent pairs.} We next discuss the problem of when the parallel transport  associated with a pair
$(A,\Phi)$ determines the pair up to gauge equivalence in the case of closed manifolds.
This problem is discussed in detail in \cite{GPSU,Pa2,Pa3,P1,P2}, but the results are only for pairs
taking values in the Lie algebra of the unitary group. The extension to the group $\GL(n,\mR)$ is non-trivial on two counts. It requires a recent extension of the Livsic theorem to arbitrary matrix groups \cite{Ka} and our new Carleman estimate for the geodesic vector field in negative curvature.

Since there is no boundary, we need to consider the parallel transport of a pair along closed geodesics.
We shall consider a simplified version of the problem, which is interesting in its own right.
We will attempt to understand those pairs
$(A,\Phi)$ with the property that the parallel transport along closed geodesics is the identity.
These pairs will be called {\it transparent} as they are invisible from the point of view of the closed geodesics of the Riemannian metric.

Let $(M,g)$ be a closed Riemannian manifold, $A$ a connection taking values in the set of real $n\times n$ matrices and $\Phi$ a potential also taking values in $\mathbb{R}^{n\times n}$.
The pair $(A,\Phi)$ naturally induces a $\GL(n,\mathbb{R})$-cocycle over
the geodesic flow $\varphi_t$ of the metric $g$ acting on the unit sphere bundle
$SM$ with projection $\pi:SM\to M$. The cocycle is defined by
\[\frac{d}{dt}C(x,v,t)=-(A(\varphi_{t}(x,v))+\Phi(\pi\circ\varphi_{t}(x,v)))C(x,v,t),\;\;\;\;\;C(x,v,0)=\mbox{\rm Id}.\]
The cocycle $C$ is said to be {\it cohomologically trivial}
if there exists a smooth function $u:SM\to \GL(n,\re)$ such that
\[C(x,v,t)=u(\varphi_{t}(x,v))u^{-1}(x,v)\]
for all $(x,v)\in SM$ and $t\in\re$. We call $u$ a trivializing function and note that two trivializing functions $u_{1}$ and $u_{2}$ (for the same cocycle) are related by $u_{2}w=u_{1}$
where $w:SM\to \GL(n,\re)$ is constant along the orbits of the geodesic flow. In particular, if $\varphi_{t}$ is transitive (i.e.\ there is a dense orbit) there is a unique trivializing function up to
right multiplication by a constant matrix in $\GL(n,\re)$.

\begin{Definition} {\rm We will say that a pair $(A,\Phi)$ is cohomologically trivial if $C$ is cohomologically trivial.
The pair $(A,\Phi)$
is said to be {\it transparent} if $C(x,v,T)=\mbox{\rm Id}$ every time
that $\varphi_{T}(x,v)=(x,v)$. }
\end{Definition}

Observe that the gauge group given by the set of smooth maps
$r:M\to \GL(n,\re)$ acts on pairs as follows:
\[(A,\Phi)\mapsto (r^{-1}dr+r^{-1}Ar,r^{-1}\Phi r).\]
This action leaves invariant the set of cohomologically trivial pairs: indeed,
if $u$ trivializes the cocycle $C$ of a pair $(A,\Phi)$, then it is easy to check that $r^{-1}u$ trivializes the cocycle of the pair $(r^{-1}dr+r^{-1}Ar,r^{-1}\Phi r)$.

Obviously a cohomologically trivial pair is transparent.
There is one important situation in which both notions agree. If $\varphi_t$ is
Anosov, then the Livsic theorem for $\GL(n,\re)$ cocycles due to Kalinin \cite{Ka} (extending the work of Livsic for the case of a cocycle taking values in a compact Lie group \cite{L1,L2}) together with the regularity results in \cite[Theorem 2.4]{NT}
imply that a transparent pair is also cohomologically trivial. 
We already pointed out that the  Anosov property is satisfied, if for example $(M,g)$ has negative curvature.

Given a cohomologically trivial pair $(A,\Phi)$, a trivializing function $u$ satisfies
\begin{equation}
(X+A+\Phi)u=0.
\label{eq:transparent}
\end{equation}
If we assume now that $(M,g)$ is negatively curved and the kernel of $X^A_+$ on $\Omega_m$ is trivial for $m \geq 1$ (i.e.\ there are no nontrivial twisted conformal Killing tensors (CKTs), see \cite{GPSU}), then the proof 
of Theorem \ref{thm_intro_attenuated_xray} implies that
$u=u_0$. If we split equation (\ref{eq:transparent}) in degrees zero and one we obtain $\Phi u_0=0$ and
$du+Au=0$. Equivalently, $\Phi=0$ and $A$ is gauge equivalent to the trivial connection.
Hence we have proved

\begin{Theorem} \label{main_thm_transparent_pairs}
Let $(M,g)$ be a closed negatively curved manifold and $(A,\Phi)$ a transparent pair.
If there are no nontrivial twisted CKTs, then $A$ is gauge equivalent to the trivial connection and
$\Phi=0$.
\label{thm:transparentpairs}
\end{Theorem}

\bibliographystyle{alpha}

\end{document}